\documentclass[letterpaper,final,twoside,leqno]{amsart}

\usepackage[normalem]{ulem}
\usepackage{mathtools}
\usepackage{mathrsfs}  
\usepackage[arrow,curve,matrix]{xy}

\usepackage{scalerel}
\usepackage{comment}

\newif\ifdraft
\draftfalse

\numberwithin{figure}{section}

\usepackage{amsmath,amsfonts,amsthm,mathrsfs}
\usepackage{amssymb}
\usepackage{mathrsfs}

\DeclareFontFamily{OMS}{rsfs}{\skewchar\font'60}
\DeclareFontShape{OMS}{rsfs}{m}{n}{<-5>rsfs5 <5-7>rsfs7 <7->rsfs10 }{}
\DeclareSymbolFont{rsfs}{OMS}{rsfs}{m}{n}
\DeclareSymbolFontAlphabet{\scr}{rsfs}

\usepackage[bookmarks]{hyperref}
\usepackage[usenames,dvipsnames]{xcolor}
\hypersetup{
    colorlinks=true,
    linkcolor={blue!50!black},
    citecolor={green!50!black},
    urlcolor={red!80!black}
}

\usepackage[alphabetic,initials]{amsrefs}

\usepackage{enumitem}

\usepackage{chngcntr}

\ifdraft
\usepackage[notcite,notref,color]{showkeys}

\definecolor{labelkey}{gray}{0.5}
\fi

\usepackage{tikz}

\usepackage{tikz-cd}
\tikzset{commutative diagrams/arrow style=math font}

\usetikzlibrary{matrix,arrows}
\newlength{\myarrowsize} 

\pgfarrowsdeclare{cmto}{cmto}{
	\pgfsetdash{}{0pt} 
	\pgfsetbeveljoin 
	\pgfsetroundcap 
	\setlength{\myarrowsize}{0.6pt}
	\addtolength{\myarrowsize}{.5\pgflinewidth}
	\pgfarrowsleftextend{-4\myarrowsize-.5\pgflinewidth} 
	\pgfarrowsrightextend{.8\pgflinewidth}
}{
	\setlength{\myarrowsize}{0.6pt} 
  	\addtolength{\myarrowsize}{.5\pgflinewidth}  
	\pgfsetlinewidth{0.5\pgflinewidth}
	\pgfsetroundjoin
	\pgfpathmoveto{\pgfpoint{1.5\pgflinewidth}{0}}
	\pgfpatharc{-109}{-170}{4\myarrowsize}
	\pgfpatharc{10}{189}{0.58\pgflinewidth and 0.2\pgflinewidth}
	\pgfpatharc{-170}{-115}{4\myarrowsize+\pgflinewidth}
	\pgfpathclose
	\pgfusepathqfillstroke
	\pgfpathmoveto{\pgfpoint{1.5\pgflinewidth}{0}}
	\pgfpatharc{109}{170}{4\myarrowsize}
	\pgfpatharc{-10}{-189}{0.58\pgflinewidth and 0.2\pgflinewidth}
	\pgfpatharc{170}{115}{4\myarrowsize+\pgflinewidth}
	\pgfpathclose
	\pgfusepathqfillstroke
	\pgfsetlinewidth{2\pgflinewidth}
}

\pgfarrowsdeclare{cmonto}{cmonto}{
	\pgfsetdash{}{0pt} 
	\pgfsetbeveljoin 
	\pgfsetroundcap 
	\setlength{\myarrowsize}{0.6pt}
	\addtolength{\myarrowsize}{.5\pgflinewidth}
	\pgfarrowsleftextend{-4\myarrowsize-.5\pgflinewidth} 
	\pgfarrowsrightextend{.8\pgflinewidth}
}{
	\setlength{\myarrowsize}{0.6pt} 
  	\addtolength{\myarrowsize}{.5\pgflinewidth}  
	\pgfsetlinewidth{0.5\pgflinewidth}
	\pgfsetroundjoin
	\pgfpathmoveto{\pgfpoint{1.5\pgflinewidth}{0}}
	\pgfpatharc{-109}{-170}{4\myarrowsize}
	\pgfpatharc{10}{189}{0.58\pgflinewidth and 0.2\pgflinewidth}
	\pgfpatharc{-170}{-115}{4\myarrowsize+\pgflinewidth}
	\pgfpathclose
	\pgfusepathqfillstroke
	\pgfpathmoveto{\pgfpoint{1.5\pgflinewidth}{0}}
	\pgfpatharc{109}{170}{4\myarrowsize}
	\pgfpatharc{-10}{-189}{0.58\pgflinewidth and 0.2\pgflinewidth}
	\pgfpatharc{170}{115}{4\myarrowsize+\pgflinewidth}
	\pgfpathclose
	\pgfusepathqfillstroke
	\pgfpathmoveto{\pgfpoint{1.5\pgflinewidth-0.3em}{0}}
	\pgfpatharc{-109}{-170}{4\myarrowsize}
	\pgfpatharc{10}{189}{0.58\pgflinewidth and 0.2\pgflinewidth}
	\pgfpatharc{-170}{-115}{4\myarrowsize+\pgflinewidth}
	\pgfpathclose
	\pgfusepathqfillstroke
	\pgfpathmoveto{\pgfpoint{1.5\pgflinewidth-0.3em}{0}}
	\pgfpatharc{109}{170}{4\myarrowsize}
	\pgfpatharc{-10}{-189}{0.58\pgflinewidth and 0.2\pgflinewidth}
	\pgfpatharc{170}{115}{4\myarrowsize+\pgflinewidth}
	\pgfpathclose
	\pgfusepathqfillstroke
	\pgfsetlinewidth{2\pgflinewidth}
}

\pgfarrowsdeclare{cmhook}{cmhook}{
	\pgfsetdash{}{0pt} 
	\pgfsetbeveljoin 
	\pgfsetroundcap 
	\setlength{\myarrowsize}{0.6pt}
	\addtolength{\myarrowsize}{.5\pgflinewidth}
	\pgfarrowsleftextend{-4\myarrowsize-.5\pgflinewidth} 
	\pgfarrowsrightextend{.8\pgflinewidth}
}{
	\setlength{\myarrowsize}{0.6pt} 
  	\addtolength{\myarrowsize}{.5\pgflinewidth}  
 	\pgfsetdash{}{0pt}
	\pgfsetroundcap
	\pgfpathmoveto{\pgfqpoint{0pt}{-4.667\pgflinewidth}}
	\pgfpathcurveto
    {\pgfqpoint{4\pgflinewidth}{-4.667\pgflinewidth}}
    {\pgfqpoint{4\pgflinewidth}{0pt}}
    {\pgfpointorigin}
	\pgfusepathqstroke
}

\newenvironment{diagram*}[2]{%
\[%
\begin{tikzpicture}[>=cmto,baseline=(current bounding box.center),%
	to/.style={->,font=\scriptsize,cap=round},%
	into/.style={cmhook->,font=\scriptsize,cap=round},%
	onto/.style={-cmonto,font=\scriptsize,cap=round},%
	math/.style={matrix of math nodes, row sep=#2, column sep=#1,%
		text height=1.5ex, text depth=0.25ex}]%
}{%
\end{tikzpicture}%
\]%
\ignorespacesafterend%
}

 \DeclareMathOperator{\Hom}{Hom}

\DeclareMathOperator{\disc}{disc}

\DeclareMathOperator{\Kosz}{Kosz}

\DeclareMathOperator{\R}{\myR}
\newcommand{\dR}{\myR}

\newcommand{\bGr}{\mathbb{G}r}

\newcommand{\wtilde}{\widetilde}

\newcommand{\hooklongrightarrow}{\lhook\joinrel\longrightarrow}

\newcommand{\theoremref}[1]{\hyperref[#1]{Theorem~\ref*{#1}}}
\newcommand{\lemmaref}[1]{\hyperref[#1]{Lemma~\ref*{#1}}}
\newcommand{\definitionref}[1]{\hyperref[#1]{Definition~\ref*{#1}}}
\newcommand{\propositionref}[1]{\hyperref[#1]{Proposition~\ref*{#1}}}
\newcommand{\conjectureref}[1]{\hyperref[#1]{Conjecture~\ref*{#1}}}
\newcommand{\corollaryref}[1]{\hyperref[#1]{Corollary~\ref*{#1}}}
\newcommand{\exampleref}[1]{\hyperref[#1]{Example~\ref*{#1}}}

\makeatletter
\let\old@caption\caption
\renewcommand*{\caption}[1]{%
  \setcounter{figure}{\value{equation}}%
  \stepcounter{equation}%
  \old@caption{#1}\relax%
}
\makeatother

\newcounter{intro}

\newtheorem{intro-conjecture}[intro]{Conjecture}
\newtheorem{intro-corollary}[intro]{Corollary}
\newtheorem{intro-theorem}[intro]{Theorem}

\newcommand{\parref}[1]{\hyperref[#1]{\S\ref*{#1}}}

\makeatletter
\newcommand*\if@single[3]{%
  \setbox0\hbox{${\mathaccent"0362{#1}}^H$}%
  \setbox2\hbox{${\mathaccent"0362{\kern0pt#1}}^H$}%
  \ifdim\ht0=\ht2 #3\else #2\fi
  }
\newcommand*\rel@kern[1]{\kern#1\dimexpr\macc@kerna}
\newcommand*\widebar[1]{\@ifnextchar^{{\wide@bar{#1}{0}}}{\wide@bar{#1}{1}}}
\newcommand*\wide@bar[2]{\if@single{#1}{\wide@bar@{#1}{#2}{1}}{\wide@bar@{#1}{#2}{2}}}
\newcommand*\wide@bar@[3]{%
  \begingroup
  \def\mathaccent##1##2{%
    \if#32 \let\macc@nucleus\first@char \fi
    \setbox\z@\hbox{$\macc@style{\macc@nucleus}_{}$}%
    \setbox\tw@\hbox{$\macc@style{\macc@nucleus}{}_{}$}%
    \dimen@\wd\tw@
    \advance\dimen@-\wd\z@
    \divide\dimen@ 3
    \@tempdima\wd\tw@
    \advance\@tempdima-\scriptspace
    \divide\@tempdima 10
    \advance\dimen@-\@tempdima
    \ifdim\dimen@>\z@ \dimen@0pt\fi
    \rel@kern{0.6}\kern-\dimen@
    \if#31
      \overline{\rel@kern{-0.6}\kern\dimen@\macc@nucleus\rel@kern{0.4}\kern\dimen@}%
      \advance\dimen@0.4\dimexpr\macc@kerna
      \let\final@kern#2%
      \ifdim\dimen@<\z@ \let\final@kern1\fi
      \if\final@kern1 \kern-\dimen@\fi
    \else
      \overline{\rel@kern{-0.6}\kern\dimen@#1}%
    \fi
  }%
  \macc@depth\@ne
  \let\math@bgroup\@empty \let\math@egroup\macc@set@skewchar
  \mathsurround\z@ \frozen@everymath{\mathgroup\macc@group\relax}%
  \macc@set@skewchar\relax
  \let\mathaccentV\macc@nested@a
  \if#31
    \macc@nested@a\relax111{#1}%
  \else
    \def\gobble@till@marker##1\endmarker{}%
    \futurelet\first@char\gobble@till@marker#1\endmarker
    \ifcat\noexpand\first@char A\else
      \def\first@char{}%
    \fi
    \macc@nested@a\relax111{\first@char}%
  \fi
  \endgroup
}
\makeatother

\usepackage{amsbsy}

\usepackage{marvosym}

\DeclareMathAlphabet{\smallchanc}{OT1}{pzc}%
                                 {m}{it}
\DeclareFontFamily{OT1}{pzc}{}
\DeclareFontShape{OT1}{pzc}{m}{it}%
             {<-> s * [1.100] pzcmi7t}{}
\DeclareMathAlphabet{\mathchanc}{OT1}{pzc}%
                                 {m}{it}
\newcommand{\mcA}{\mathchanc{A}}
\newcommand{\mcB}{\mathchanc{B}}

\newcommand{\mcR}{\mathchanc{R}}

\newcommand{\sA}{\mathscr{A}}
\newcommand{\sB}{\mathscr{B}}

\newcommand{\sE}{\mathscr{E}}
\newcommand{\sF}{\mathscr{F}}
\newcommand{\sG}{\mathscr{G}}

\newcommand{\sL}{\mathscr{L}}
\newcommand{\sM}{\mathscr{M}}
\newcommand{\sN}{\mathscr{N}}
\newcommand{\sO}{\mathscr{O}}

\newcommand{\sW}{\mathscr{W}}

\newcommand{\sfA}{{\sf A}}
\newcommand{\sfB}{{\sf B}}

\newcommand{\sfK}{{\sf K}}

\newcommand{\bC}{\mathbb{C}}

\newcommand{\bF}{\mathbb{F}}
\newcommand{\bG}{\mathbb{G}}

\newcommand{\bN}{\mathbb{N}}

\newcommand{\bP}{\mathbb{P}}
\newcommand{\bQ}{\mathbb{Q}}

\newcommand{\bZ}{\mathbb{Z}}

\newcommand{\frF}{\mathfrak{F}}

\newcommand{\frL}{\mathfrak{L}}

\newcommand{\frQ}{\mathfrak{Q}}

\newcommand{\frS}{\mathfrak{S}}

\newcommand{\frc}{\mathfrak{c}}

\newcommand{\fre}{\mathfrak{e}}
\newcommand{\frf}{\mathfrak{f}}

\newcommand{\frh}{\mathfrak{h}}
\newcommand{\fri}{\mathfrak{i}}

\newcommand{\frl}{\mathfrak{l}}

\newcommand{\fro}{\mathfrak{o}}

\newcommand{\frr}{\mathfrak{r}}

\newcommand{\frt}{\mathfrak{t}}

\DeclareSymbolFont{largesymbolsA}{U}{jkpexa}{m}{n}
\SetSymbolFont{largesymbolsA}{bold}{U}{jkpexa}{bx}{n}
\DeclareMathSymbol{\varprod}{\mathop}{largesymbolsA}{16}
\makeatletter
\newcommand{\LeftEqNo}{\let\veqno\@@leqno}
\makeatother
\newcommand{\properideal}%
        {\subsetneq}

\newcommand{\wt}{\widetilde}

\newcommand{\leteq}{\colon\!\!\!=}

\DeclareMathOperator{\codim}{codim}

\DeclareMathOperator{\Ex}{Ex}

\DeclareMathOperator{\Gr}{{Gr}}

\DeclareMathOperator{\id}{{id}}

\DeclareMathOperator{\mor}{{Mor}}

\DeclareMathOperator{\Ob}{{Ob}}

\DeclareMathOperator{\ob}{{Ob}}

\DeclareMathOperator{\rank}{{rank}}

\DeclareMathOperator{\sgn}{{sgn}}

\DeclareMathOperator{\supp}{{supp}}

\newcommand{\factor}[2]{\left. \raise .2em\hbox{\ensuremath{#1}\vphantom{$I^d$}}
\hskip -.1em \right/ \hskip -.4em \raise -.3em\hbox{\ensuremath{#2}}}%
\newcommand\mtimes[3]{{\varprod_{#1}^{#2}}_{\raise 1ex \hbox{\scriptsize #3}}}%

\newcommand{\myR}{{\mcR\!}}

\newcommand{\kdot}{{{\,\begin{picture}(1,1)(-1,-2)\circle*{2}\end{picture}\,}}}

\def\dimcoh#1.#2.#3.{h^{#1}(#2,#3)}
\def\hypcoh#1.#2.#3.{\mathbb H_{\vphantom{l}}^{#1}(#2,#3)}
\def\loccoh#1.#2.#3.#4.{H^{#1}_{#2}(#3,#4)}
\def\dimloccoh#1.#2.#3.#4.{h^{#1}_{#2}(#3,#4)}
\def\lochypcoh#1.#2.#3.#4.{\mathbb H^{#1}_{#2}(#3,#4)}
\def\seslong#1.#2.#3.{0  \longrightarrow  #1   \longrightarrow 
 #2 \longrightarrow #3 \longrightarrow 0} 
\def\sesshort#1.#2.#3.{0
 \rightarrow #1 \rightarrow #2 \rightarrow #3 \rightarrow 0}
\def\Iff#1#2#3{
\hfil\hbox{\hsize =#1
\vtop{\noin #2}
\hskip.5cm 
\lower.5\baselineskip\hbox{$\Leftrightarrow$}\hskip.5cm
\vtop{\noin #3}}\hfil\medskip}
\newcommand{\union}\cup
\newcommand{\intersect}\cap
\newcommand{\Union}\bigcup
\newcommand{\Intersect}\bigcap
\def\myoplus#1.#2.{\underset #1 \to {\overset #2 \to \oplus}}

\newcommand{\resto}[1]{\raise -.5ex\hbox{$\vert$}_{#1}}

\def\qis{\,{\simeq}_{\text{qis}}\,}

\newcommand\noin{\noindent}
\newcommand\tightdot{\!\cdot\!}

   [2017/06/13 v0.1 finetuning mabliautoref and theorem setup]

\RequirePackage{hyperref}

\newcommand{\sectionsize}{} 
\newcommand{\theoremsize}{} 

\renewcommand{\subsectionautorefname}{\sectionsize\sf \subsectionautorefname}

\makeatletter
\@ifdefinable\equationname{\let\equationname\equationautorefname}
\def\equationautorefname~#1\@empty\@empty\null{\protect{\theoremsize
    (#1\@empty\@empty\null)}}%
\@ifdefinable\AMSname{\let\AMSname\AMSautorefname}
\def\AMSautorefname~#1\@empty\@empty\null{
  ( #1\@empty\@empty\null)}%
\@ifdefinable\itemname{\let\itemname\itemautorefname}
\def\itemautorefname~#1\@empty\@empty\null{\theoremsize{%
    {#1}}\@empty\@empty\null%
}%
\makeatother

\RequirePackage{amsthm}
\RequirePackage{aliascnt}
\newcommand{\basetheorem}[3]{%
    \newtheorem{#1}{#2}[#3]
    \newtheorem*{#1*}{#2}
    \expandafter\def\csname #1autorefname\endcsname{#2}
}%
\newcommand{\maketheorem}[3]{%
    \newaliascnt{#1}{#2}
    \newtheorem{#1}[#1]{\theoremsize #3}
    \aliascntresetthe{#1}
    \expandafter\def\csname #1autorefname\endcsname{#3}
    \newtheorem{#1*}{#3}
}%
\newcommand{\baseremark}[3]{%
    \newtheorem{#1}{#2}{#3}
    \newtheorem*{#1*}{#2}
    \expandafter\def\csname #1autorefname\endcsname{#2}
}%
\newcommand{\makeremark}[3]{%
    \newaliascnt{#1}{#2}
    \newtheorem{#1}[equation]{#3}
    \aliascntresetthe{#1}
    \expandafter\def\csname #1autorefname\endcsname{\theoremsize\sf #3}
    \newtheorem{#1*}{#3}
}%
\theoremstyle{plain}   

\basetheorem{theorem}{Theorem}{section}

\theoremstyle{definition}    

\theoremstyle{plain}
\maketheorem{prop}{theorem}{Proposition}
\maketheorem{cor}{theorem}{Corollary}
\maketheorem{lem}{theorem}{Lemma}
\maketheorem{conj}{theorem}{Conjecture}

\newtheorem{proposition}[prop]{Proposition}
\newtheorem{corollary}[cor]{Corollary}
\newtheorem{lemma}[lem]{Lemma}

\theoremstyle{definition}
\maketheorem{defini}{theorem}{Definition}
\newtheorem{num}[defini]{}
\newtheorem{definition}[defini]{Definition}
\newtheorem{example}[theorem]{Example}
\theoremstyle{remark}
\maketheorem{rem}{theorem}{Remark}
\newtheorem{remark}[rem]{Remark}
\newtheorem{subrem}[equation]{Remark}
\maketheorem{notat}{theorem}{Notation}
\maketheorem{notr}{theorem}{Notation-Remark}
\newtheorem{notation}[notat]{Notation}
\newtheorem{not-rem}[notr]{Notation-Remark}
\maketheorem{clam}{theorem}{Claim}
\newtheorem{claim}[clam]{Claim}

\maketheorem{fact}{theorem}{Fact}

\setlist[enumerate]{itemsep=3pt,topsep=3pt,leftmargin=2em,label={(\roman*)}}

\numberwithin{equation}{theorem}

\renewcommand\thesubsection{\thesection.\Alph{subsection}}

\newcommand\af{a_{\wtilde f}}

\begin{document}

\title
{Hodge sheaves underlying flat projective families}

\author{S\'andor J Kov\'acs} \address{S\'andor Kov\'acs, Department of Mathematics,
  University of Washington, Box 354350, Seattle, Washington, 98195, U.S.A}
\email{\href{mailto:skovacs@uw.edu }{skovacs@uw.edu }}
\urladdr{\href{http://sites.math.washington.edu/~kovacs/current/}
  {http://sites.math.washington.edu/~kovacs/current/}}

\author{Behrouz Taji} \address{Behrouz Taji, School of Mathematics and Statistics,
The University of New South Wales Sydney, NSW 2052 Australia}
\email{\href{mailto:b.taji@unsw.edu.au}{b.taji@unsw.edu.au}}
\urladdr{\href{http://www.maths.usyd.edu.au/u/behrouzt/}
  {http://www.maths.usyd.edu.au/u/behrouzt/}}

\thanks{S\'andor Kov\'acs was supported in part by NSF Grants DMS-1565352,
  DMS-1951376, and DMS-2100389.}




\setlength{\parskip}{0.19\baselineskip}


\begin{abstract}
  We show that, for any fixed weight, there is a natural system of Hodge sheaves, whose
  Higgs field has no poles, arising from a flat projective family of varieties
  parametrized by a regular complex base scheme, extending the analogous classical
  result for smooth projective families due to Griffiths.  As an application, based
  on positivity of direct image sheaves, we establish a criterion for base spaces of
  rational Gorenstein families to be of general type.
  A key component of our arguments is centered around the construction of
  derived categorical objects generalizing relative logarithmic forms for smooth
  maps and their functorial properties.
\end{abstract}
  
\maketitle

\tableofcontents

\section{Introduction and main results}
\label{sect:Section1-Introduction}

In a series of seminal works \cite{Gri68I}, \cite{Gri68II}, and \cite{Gri70},
Griffiths established that a degeneration of polarized Hodge structures (of fixed
weight) in a smooth projective family $f: X\to B$ induces
\begin{enumerate}
\item\label{item:15} a flat bundle $(\mathcal V, \nabla)$ on $B$, equipped with a
\item\label{item:16} system of Hodge bundles $(\sE, \theta)$, and a
\item\label{item:17} natural analytic data defined by a harmonic metric.
\end{enumerate}

Following this discovery, these fundamental results were later successfully developed
further in two major new directions.  Through nonabelian Hodge theory,
Simpson~\cite{MR1179076} and Mochizuki~\cite{MR2310103} established topological
characterizations of \autoref{item:15}, \autoref{item:16}, and \autoref{item:17},
regardless of a geometric origin (a smooth projective family).  In a different
direction by replacing \autoref{item:15} and the Hodge filtration by filtered
holonomic $\mathcal D$-modules, Hodge modules were introduced by
Saito~\cite{MR1047415} as a generalization of variations of Hodge structures (VHS for
short) for non-smooth families. None of these two general theories will be used in
this paper.

For smooth projective families we know that the direct summands of $\sE$ are
represented by the cohomology of sheaves of relative K\"ahler forms; an
algebro-geometric datum.  From a geometric point of view the existence of
\autoref{item:16} for non-smooth families and one that is similarly of
algebro-geometric origin is of special interest,\!\footnote{See for example
  \cite{Gri84}*{II, VII}, \cite{Ste76}, \cite{Steen76}, and more recently
  \cite{Vie-Zuo03a} and \cite{Taj20}.}  cf.~\ref{ssec:sing-famil}.

Our first goal in this paper is to establish that in fact any flat family of
projective varieties gives rise to systems of Hodge sheaves (with no poles), as soon
as the base of the family is smooth. Moreover, we will see that, similar to the
smooth case, they arise from cohomology of objects---in the derived category---that
play the role of relative K\"ahler forms for non-smooth families,
cf.~\autoref{sect:Section3-DBExtension}.

\begin{theorem}\label{thm:DBI}
  Let $f:X\to B$ be a flat projective morphism of reduced complex schemes with
  connected fibers, where $B$ is a smooth complex variety. Further let $w\in\bN$,
  $0\leq w \leq \dim(X/B)$. Then, there exists a functorial system of reflexive Hodge
  sheaves $(\overline \sE = \bigoplus_{i=0}^w \overline \sE_{i}, \overline \theta)$,
  $\overline \theta: \overline \sE_{i} \to \Omega^1_B \otimes \overline \sE_{i+1}$,
  of weight $w$ on $B$.  If in addition $X$ has only rational singularities and
  $w=\dim(X/B)$, then $\overline \sE_{0}\simeq (f_*\omega_{X/B})^{**}$.
\end{theorem}

\begin{subrem}
  See \autoref{def:system} and the subsequent paragraph for the definition of a
  system of reflexive Hodges sheaves and Subsection~\ref{ssec:funct} and
  \autoref{thm:funct} for the functoriality of such a system.
\end{subrem}

\noindent
Our next goal is to compare these Hodge sheaves to the logarithmic system
$(\sE^{\bf 0}, \theta^0)$ underlying the \emph{Deligne canonical extension}
\cite{Deligne70}*{I.5.4} $\mathcal V^{0}$ of integral variation of Hodge structures
of weight $w$ for a smooth model. 
Here we are following the standard parabolic notation for extensions of $\mathcal
V$. That is, for a tuple $\mathbf{\beta} = (\beta_i)_i$ of real numbers $\beta_i$,
$j: B \setminus D_{\wtilde f} \to B$ is the inclusion map and
$D_{\wtilde f} = \sum D_{\wtilde f}^i$ is the discriminant locus of $\wtilde f$
(which is assumed to be normal crossing) is defined as follows.  The sequence of
holomorphic bundles $\mathcal V^{\bf \beta}$ is the decreasing filtration of
$j_* \mathcal V$, defined by the lattice with respect to which
$\rm{res}(\nabla)|_{D_{\wtilde f}^i}$ has eigenvalues in $[\beta_i,
\beta_{i+1})$. Throughout this paper $[\beta_i, \beta_{i+1})$ is fixed to be equal
for all $i$. 
More precisely, given a {suitable resolution} $\pi: \wtilde X \to X$ and the
resulting family $\wtilde f: \wtilde X\to B$, we show that there is a nonnegative
integer $\af$, that encodes how singular the \emph{family} $f$ is and measures the
difference between $\overline \sE$ and $\sE^{0}$, where $\sE^{0}$ is the Deligne
extension of the integral VHS associated to the smooth locus of $\wtilde f$.

\begin{theorem}\label{thm:DBI'}
  In the setting of \autoref{thm:DBI}, let $\pi: \wtilde X\to X$ be a good resolution
  with respect to $f$, and denote the resulting morphism by
  $\wtilde f: \wtilde X \to B$.  Further let $D_{\wtilde f}$ denote the divisorial
  part of the discriminant locus of $\wtilde f$, and assume that it is an snc divisor
  on $B$. (This can be achieved by base changing to an embedded resolution over $B$.)

  Then, there exists an integer
  \begin{equation}
    \label{eq:13}
    0\leq \af\leq \dim (X) 
  \end{equation}
  for which we have an inclusion of systems of equal weights
  \begin{equation}\label{eq:Inclusion}
    (\overline \sE, \overline \theta) \subseteq   \big( \sE^{0} , \theta^{0}
    \big) (\af \tightdot D_{\wtilde f} )  
    \simeq_{\mathbb C^{\infty}}  \mathcal V^{-\af}  ,
  \end{equation}
  This isomorphism 
  is in the category of smooth bundles.
\end{theorem}

Here, a \emph{good resolution with respect to $f$} means a desingularization $\pi$
for which $\wtilde f^*D_{\wtilde f}$ has simple normal crossing (snc) support. For a
more detailed and precise definition see \autoref{def:good-hyper-res}.  In
\autoref{thm:DBI'} and in the rest of this article
$(\sE^0, \theta^0) (\af \tightdot D_{\wtilde f})$ denotes the naturally induced
system of Hodge sheaves defined by
\[
  ( \sE^0 \otimes \sO_X( \af \tightdot D_{\wtilde f} ) , \theta^0 \otimes \id).
\]
The integer $\af$ will be called the \emph{discrepancy of $f$ with respect to
  $\pi:\wtilde X\to X$}.  Note that $\af$ can be interpreted as a measure of
degeneration in the family; the smaller this integer, the milder the singularity of
the degeneration.  In particular when $f$ is smooth, we have $\af=0$.

\subsection{Functoriality}
\label{ssec:funct}
An important feature of the construction of $(\overline \sE, \overline \theta)$ in 
\autoref{thm:DBI} is its functoriality. More precisely, 
one can consider a category $\mathfrak{Fam}(n,d)$ of morphisms $f: X\to B$
as in \autoref{thm:DBI}, 
where $\dim (X)=n$ and $\dim (B) =d$,
and a category $\mathfrak{Hodge}(d, w)$ of systems of Hodge sheaves of weight $w$
(see \ref{ssect:funct} for the precise definitions). 
The system $(\overline \sE, \overline \theta)$ in \autoref{thm:DBI} then gives rise 
to a functor between these two 
categories.

\begin{theorem}\label{thm:funct}
Let $n, d, w \in \bN$. There exists a functor $\chi_w: \mathfrak{Fam}(n,d) \to \mathfrak{Hodge}(d,w)$ 
defined by $\chi_w(f:X\to B) = (B, \overline \sE, \overline \theta)$, where $(\overline \sE, \overline \theta)$ 
is the system 
in \autoref{thm:DBI}. Furthermore, for $(f:X\to B)\in \Ob(\mathfrak{Fam}(n,d))$ and 
any open subset $V\subseteq B$, we have 

$$
\chi_w(f: X\to B)|_{V} = \chi_w( f_V: X_{V} \to V ) ,
$$
where $X_V:= f^{-1}(V)$ and $f_V:= f|_{X_V}$. 

\end{theorem}

\subsection{Singular families of varieties with base schemes of general type}
\label{ssec:sing-famil}
Viehweg conjectured that for a projective morphism $f:X\to B$ of smooth projective
varieties $X$ and $B$ with connected fibers and $D$ denoting the (divisorial part of)
the discriminant locus of $f$, if $f$ has maximal variation, and its smooth fibers
are canonically polarized, then $(B,D)$ is of log general type, i.e., $\omega_B(D)$
is big.

This conjecture generated considerable interest and for several years. It was finally
resolved, and in fact generalized, through the culmination of work of several authors
including \cite{Kovacs96}, \cite{Kovacs00a}, \cite{Oguiso-Viehweg01},
\cite{Vie-Zuo01}, \cite{Kovacs02}, \cite{VZ02}, \cite{Kovacs03b},
\cite{KK08},\cite{KK08b}, \cite{KK10}, \cite{MR2871152}, \cite{PS15}, \cite{Taj18}
and \cite{CP16}.

In higher dimensions, the minimal model program taught us that when positivity of
canonical sheaves are involved, it is desirable to try to extend results to mildly
singular cases. So, it is natural to ask whether Viehweg's conjecture extends to
families of minimal models.  The simple answer is that the desired positivity fails
already, if one allows Gorenstein terminal singularities, arguably the mildest
possible. In particular, the conjecture fails for Lefschetz pencils,
cf.~\ref{sssec:order-poles-lefsch}.

This could be interpreted as a sign that there is no reasonable generalization of
Viehweg's conjecture to singular families. However, here we offer a potential way to
remedy the situation. Before we state that generalization, first recall that the
initial step in the proof of essentially any result connected to Viehweg's conjecture
has been a related result (which itself was a culmination of work of Fujita,
Kawamata, Koll\'ar, and Viehweg), which states that if a family of varieties of
general type has maximal variation, then the line bundle $\det f_*\omega^m_{X/B}$ is
big, i.e., has maximal Kodaira dimension. Reformulating Viehweg's conjecture in terms
of the bigness of this line bundle has the advantage that it allows one to remove the
condition that the fibers would be canonically polarized or even of general type. So,
by including this initial step of the proof in the conjecture itself one may rephrase
Viehweg's conjecture in terms of $\det f_*\omega^m_{X/B}$ being big, instead of
requiring maximal variation and that the fibers be of general type. This formulation
allows one to quantify (to some extent) the starting assumption for singular families
and ask that not only $\det f_*\omega^m_{X/B}$ be big, but that it should be big
compared to something else.

As an application of \autoref{thm:DBI'}, we show that for Gorenstein families it is
possible to obtain a result similar to Viehweg's conjecture along the lines outlined
above. This requires that we take into account how singular the family is. More
precisely, we show that if $\det f_*\omega^m_{X/B}$ is positive \emph{enough} to
balance the \emph{discrepancy} of the family (discussed above), then the base of the
family is indeed necessarily of (log) general type.

\begin{theorem}
  \label{thm:Cons}
  Let $X$ and $B$ be projective varieties and $f:X\to B$ a flat family of
  geometrically integral varieties with only Gorenstein Du~Bois singularities, such
  that $B$ is smooth and the generic fiber of $f$ has rational singularities.
  Further let $D,D'\subset B$ be effective divisors such that $D+D'= D_{\wtilde f}$
  and let $r_m:= \rank(f_*\omega^m_{X/B})$.  If
  $(\det f_*\omega^m_{X/B})( - m r_m \dim (X)\cdot D)$ is big, then $(B, D')$ is of
  log-general type.
\end{theorem}

\begin{rem}
  Observe that this theorem includes Viehweg's conjecture: Assuming that the $n$-dimensional 
 variety  $X$ is
  smooth and taking $D=0$. This also shows that this statement is
  stronger than Viehweg's conjecture even in the original situation. Viehweg's
  conjecture predicted that maximal variation of the family implies that the base is
  of log general type with respect to the boundary divisor chosen to be the
  codimension one part of the discrepancy locus. \autoref{thm:Cons} says that this
  can be improved: any divisor $D$ that's part of the discrepancy locus and has the
  property that $(\det f_*\omega^m_{X/B})( - m r_m n\cdot D)$ is big may be
  subtracted from the boundary divisor. In other words, if the pushforward of a
  pluricanonical sheaf is ``bigger'' than any part of the discrepancy locus, then one
  obtains that the base is of log general type with a \emph{smaller} boundary
  divisor. In the extreme case that $(\det f_*\omega^m_{X/B})( - m r_m n\cdot D_f)$
  is big, this means that the base itself has to be of general type.

  This strengthening of Viehweg's conjecture is new even in the case when $X$ is
  smooth, but in \autoref{thm:Cons} we actually allow a singular $X$.  This result
  could also be used in a reverse way to give a lower bound on discrepancy divisors
  of some families, or the discrepancy divisor of any of their resolutins. For
  instance, one obtains a bound for the notorious Lefschetz pencils.

  Notice further, that in \autoref{thm:Cons} there is no assumption on the Kodaira
  dimension of the fibers, which is another way this result is much more general than
  Viehweg's original conjecture.

  On the other end of the spectrum, \autoref{thm:Cons} implies that for every flat
  rational Gorenstein family we have the following implication: 
  \[
    \kappa\big( (\det f_*\omega^m_{X/B})( - mr_m \dim X \cdot D_{\wtilde f}) \big) =
    \dim (B)\implies \kappa(B) = \dim B.
  \]

  Finally note, that if $f:X\to B$ is a KSB-stable family, with $X$ Gorenstein and $B$
  smooth, such that the general fiber of $f$ has rational singularities, then $f$
  satisfies the assumption on the singularities in \autoref{thm:Cons} by \cite{KK10b}
  and hence \autoref{thm:Cons} applies to such families.
\end{rem}

\subsection{Hyperfiltered logarithmic forms in the derived category}
Inspired by the works of Katz-Oda~\cite{KO68} our construction of
$(\overline \sE, \overline \theta)$ in \autoref{thm:DBI} fundamentally depends on the
existence of a filtration, or more precisely the \emph{Koszul filtration}, that is
naturally available for K\"ahler forms of smooth families. In the absence of such
objects with analogous properties for singular families, we pass on to the derived
category $D^b(X)$, where an appropriate hyperfiltration $\bF$ (in the derived sense,
see \autoref{def:Koszul-filtration}) was constructed in \cite{Kovacs05a} and applied
to the complex of Deligne-Du~Bois forms,
which are objects in the bounded derived category of coherent sheaves of $X$. These
objects have similar cohomological properties to the sheaves $\Omega_X^p$ in the
smooth case (see \autoref{def:cx-db-forms}). For more details regarding the complexes
of Deligne-Du~Bois forms see \autoref{sec:rel-db-cxs}, \cite{DuBois81},
\cite{GNPP88}, \cite{Kovacs97b}*{3.1}, \cite{Kovacs-Schwede11}*{\S 4}, and
\cite{PetersSteenbrinkBook}*{7.3.1}.

For smooth projective families, through the Hodge-to-de Rham spectral sequence
degeneration, one uses holomorphicity and transversality properties of $\nabla$ to
extract an underlying system of Hodge bundles.  When $f$ is singular, in the absence
of such tools, including a filtered relative de Rham complex satisfying good
degeneration properties, analogous results cannot be similarly established by the
same methods.

To circumvent this difficulty we construct the \emph{complex of logarithmic Du Bois
  $p$-forms} $\underline\Omega^p_X (\log \Delta) \in \Ob D^b(X)$, which can be endowed
with the structure of a Koszul-type hyperfiltration $\bF_f$ using the construction in
\cite{Kovacs05a}.  Moreover, we show that for a morphism of snc pairs
$f: (X,\Delta)\to (B,D)$ (see \autoref{def:families-pairs})
$(\underline \Omega^p_X (\log \Delta), \bF^{\kdot}_f)$ is filtered quasi-isomorphic
to $(\Omega^p_X (\log \Delta), F_{K}^{\kdot})$, where $F^{\kdot}_{K}$ is the usual
Koszul filtration.  See \autoref{thm:log-DB} for details.

In \autoref{sec:rel-db-cxs}, we show that this hyperfiltration is functorial, and
using this functoriality we establish a natural filtered pullback map from
$\underline \Omega^p_X$ to $\underline \Omega^p_{\wtilde X} (\log \Delta)$, twisted
with a well-understood line bundle that encodes the singularity of the family $f$ in
terms of $\af$ (the \emph{discrepancy} of $f$ with respect to
$\pi:\wtilde X\to X$).  On the other hand, $(\sE^0, \theta^0)$ is determined by
$(\Omega^p_{\wtilde X}(\log \Delta), F^{\kdot}_K)$ by \cite{Ste76} and \cite{KO68}.
Now, the fact that, for each $0\leq p \leq \dim X/B$, the two filtered objects
$\underline\Omega^p_X$ and $\underline \Omega^p_{\wtilde X}(\log \Delta)$ are
functorially related then leads to the formation of
$(\overline\sE, \overline \theta)$ compatible with $(\sE^{\bf 0}, \theta^0)$ (in the
sense of \autoref{eq:Inclusion}), endowing the former with the structure of a system
of Hodge sheaves.


\subsection{Singularities of Higgs fields underlying VHSs of geometric origin} 
The Gauss-Manin connection $\nabla$ arising from a smooth projective family extends
to $\mathcal V^0$ with only logarithmic poles due to its integrability, as shown by
Manin \cite{Man65} and Deligne \cite{Deligne70}*{I.5.4}.  However, in general such
flat connections do not have trivial local monodromy at infinity and thus their
singularities are often not removable.  On the other hand, for a polarized VHS over a
punctured polydisk with unipotent monodromy, the Hodge filtration extends to a
holomorphic filtration of $\mathcal V^0$ by Schmid's Nilpotent Orbit Theorem
\cite{Sch73} and the results of Cattani-Kaplan-Schmid \cite{CKS}. It follows that the
poles of $(\sE^0, \theta^0)$, as a Higgs bundle, are at worst logarithmic. In fact,
at least over a smooth quasi-projective variety, and for a suitable choice of
extension, the same is true for all tame harmonic bundles \cite{MR2283665}*{22.1}.
As a direct consequence of \autoref{thm:DBI} and \autoref{thm:funct} we can show that there is always an
extension of the Higgs bundle $(\sE, \theta)$ underlying $(\mathcal V, \nabla)$ with
zero residues\footnote{In some sense this gives an optimal algebro-geometric
  realization of the fact that $\theta$ is nilpotent.}. In other words, $\theta$ has
removable singularities. We make this point more precise in the following remark.

\begin{remark}
  \label{thm:DBII}
  In the setting of \autoref{thm:DBI}, further assume that $X$ is smooth.  Let $D_f$
  denote the divisorial part of the discriminant locus of $f$ and assume that $D_f$
  and $f^*D_f$ have simple normal crossing support.  Then, for any fixed weight, the
  system $(\overline \sE, \overline \theta)$ in \autoref{thm:DBI} is an extension of
  the Hodge bundle of the same weight underlying the VHS of the smooth locus of $f$,
  that is
  $ ( \overline \sE, \overline \theta )|_{B\setminus D_f} \cong (\sE, \theta)$.
\end{remark}
  
In the context of \autoref{thm:DBII}, we call $(\overline \sE, \overline \theta)$ a
\emph{derived extension}. We note that the inclusion
  $(\overline \sE, \overline\theta) (-\af\tightdot D_f) \subseteq
  (\sE^0, \theta^0)$ guarantees that there is always a subextension of
  $(\sE^0, \theta^0)$ with vanishing residues.

  \autoref{thm:DBII} can be interpreted as providing an analytic criterion for
  detecting when a VHS is not of geometric origin (and similarly for a complex VHS in
  the sense of~\cite{MR1040197}*{p.~868}).

\begin{corollary}
  Let $B$ be a smooth complex variety, $D\subseteq B$ a simple normal crossing
  divisor, and $(\mathcal V, \nabla, \sE = \bigoplus\sE_i, \theta)$ an abstract real
  VHS on $B\setminus D$.  If the given VHS is of geometric origin, then the
  singularity of $\theta$ is removable, i.e., there exists a reflexive Hodge sheaf
  $(\sE',\theta')$ on $B$, with $\theta': \sE'\to \Omega^1_B \otimes \sE'$, such that
  $(\sE', \theta')|_{B\setminus D} \cong (\sE,\theta)$.
\end{corollary}

\subsubsection{Order of poles for Lefschetz pencils}\label{sssec:order-poles-lefsch} 
We emphasize that Deligne extensions (or their underlying Hodge bundle) have logarithmic poles even in the case of very
mild degenerations such as Lefschetz pencils of non-hyperelliptic curves (a
particular instance of a stable family of curves).  To see this, one may use the
following observation.  Note that $f_*\omega_{X/\mathbb P^1}$ is the first graded
piece of the Hodge sheaves underlying the Deligne extension
$(\sE^{0}= \sE_1^{0} \oplus \sE_2^0, \theta^{0})$ of $\R^1f^\circ_* \bC_X$,
where $f^\circ$ denotes the smooth locus of $f$.  By the weak positivity of
$f_*\omega_{X/\bP^1}$ (see for example \cite{Viehweg95}) we know that every rank-one
direct summand $\sL_j$ in the splitting $f_*\omega_{X/\bP^1}\cong \bigoplus \sL_j$ is
nef.
Over the smooth locus of $f$, the Higgs field $\theta^0$ is locally equal to the
derivative of the period map, so by the local Torelli theorem $\theta^0\neq 0$.
Therefore, we have $\theta^0(\sL_j)\neq 0$, for some $j$.  Now if $\theta^0$ had no
poles, by applying $\theta^0$ to $\sL_j$ and using the weak negativity\footnote{A
  weakly negative sheaf is one whose dual is weakly positive (see \cite{Viehweg95}*{2.3}).}  of the kernel of
$\theta^0$, cf.~\cite{Gri84}, \cite{MR1040197} or~\cite{Zuo00}, we would get an
induced injection
\[
  \sL_j \longrightarrow   \Omega^1_{\bP^1} \otimes \sE_2^0.
\]
It follows, again from the weak negativity of $\sE_2^0$, that after taking determinants
there is an injection
\[
\sL_j^t \otimes (\det \sE_2^0)^{-1} \longrightarrow  (\Omega^1_{\bP^1})^{\otimes t},
\]
where $t=\rank(\sE_2^0)$ and $(\det \sE_2^0)^{-1}$ is nef. But this is absurd,
showing that indeed $\theta^0$ must have poles.

\subsection{Acknowledgments} The second named author owes a special debt of gratitude
to G\'abor Sz\'ekelyhidi for his support and encouragements. We are grateful to Erwan
Rousseau for his help during our joint visit to CIRM.  We would also like to thank
Geordie Williamson and the referee for helpful comments and suggestions.

\section{Preliminary definitions and notation}\label{sec:definitions-notation}

\subsection{Families of pairs}

The study of \emph{pairs} or \emph{log varieties} have led to many advances in
birational geometry and moduli theory. For the questions investigated here a simple
version of pairs will suffice, namely we will restrict to the case when the boundary
divisor is reduced.

\begin{definition}\label{def:families-pairs}
  A \emph{reduced pair} $(X,\Delta)$ consists of a normal scheme $X$ and an effective
  reduced divisor $\Delta\subset X$.
  An \emph{snc pair} is a reduced pair $(X,\Delta)$ such that $X$ is smooth and
  $\Delta$ is an snc divisor. 
  A \emph{morphism of (reduced) pairs} $f:(X,\Delta)\to (B,D)$ is a morphism
  $f:X\to B$ of normal schemes such that $\supp\Delta\supseteq f^{-1}(\supp D)$.
  Assuming that $D$ is $\bQ$-Cartier, we will use the notation
  $f^{-1}D\leteq (f^*D)_{\mathrm{red}}$ to denote the \emph{reduced preimage} of $D$.
  Using this notation the above criterion can be replaced by $\Delta\geq f^{-1}D$.
  A \emph{morphism of snc pairs} is a morphism of reduced pairs
  $f:(X,\Delta)\to (B,D)$ such that both $(X,\Delta)$ and $(B,D)$ are snc pairs.

  Consider a morphism of reduced pairs $f:(X,\Delta)\to (B,D)$ and a decomposition
  $\Delta=\Delta_v+\Delta_h$ into vertical and horizontal parts, i.e., such that
  $\codim_Bf(\Delta_v)\geq 1$ and that $f|_{\Delta_0}$ dominates $B$, for any
  irreducible component $\Delta_0\subseteq \Delta_h$. Using this decomposition, we
  call a morphism of snc pairs $f:(X,\Delta) \to (B, D)$ an \emph{snc morphism}, if
  $f$ is flat, $\Delta_v=f^{-1}D$ and $f|_{X\setminus \Delta_v}$ is smooth.
  %
  The composition of two (snc) morphisms of pairs is also a (snc) morphism of pairs.

  Further note that the term ``morphism of pairs'' does not have a standard usage and
  it may be used to refer to a somewhat different situation by other authors. We
  added the extra word ``reduced'' to remind the reader of this potential
  difference. We are still not claiming that this definition is standard. We believe
  that an established standard usage of this phrase does not exist at this time.
\end{definition}

\begin{definition}\label{def:Koszul-filtration}
  Let $f:(X,\Delta)\to (B,D)$ be an snc morphism. Then, after removing a codimension
  $2$ subset of $B$, there exists a short exact sequence of locally free sheaves,
  \[
    \xymatrix{%
      0 \ar[r] & f^*\Omega^1_B(\log D) \ar[r] & \Omega^1_X(\log \Delta) \ar[r] &
      \Omega^1_{X/B}(\log \Delta) \ar[r] & 0.  }
  \]
  For each $0\leq q\leq \dim(X/B)$ this induces a descending filtration, called the
  \emph{Koszul filtration} and denoted by $F_K^\kdot\Omega^q_X(\log \Delta)$ such
  that the associated graded quotients of the filtration satisfy that
  \begin{equation}
    \label{eq:3}
    \Gr_{F_K^{\kdot}}^j\Omega^q_X(\log \Delta)\leteq \factor {F_K^j\Omega^q_X(\log
      \Delta)}{F_K^{j+1}\Omega^q_X(\log \Delta)}\simeq f^*\Omega^{j}_B(\log B)
       \otimes \Omega^{q-j}_{X/B}(\log \Delta) .
  \end{equation}
\end{definition}

The reader is referred to \cite{GNPP88} for the definition of simplicial and cubic schemes. 
In this paper a \emph{hyperresolution} will mean a cubic scheme, all of whose 
entries are smooth schemes of finite type over $\bC$.

\begin{definition}\label{def:good-hyper-res}
  Let $(X,\Delta)$ be a reduced pair.  A \emph{good resolution} (or \emph{log
    resolution}) of $(X,\Delta)$, is a proper birational morphism of pairs
  $g:(Y,\Gamma)\to (X,\Delta)$ such that $X$ is quasi-projective, the exceptional set
  $E\leteq \Ex(g)$ of $g$ is a divisor, $\Gamma=g^{-1}_*\Delta+E$ and $(Y,\Gamma)$ is
  an snc pair.

  Let $(X,\Delta)$ be a reduced pair and $f:X\to B$ a morphism. A \emph{good
    resolution} of $(X,\Delta)$ \emph{with respect to $f$} is a good resolution
  $g:(Y,\Gamma)\to (X,\Delta)$ such that in addition to the above, $\Gamma+g^{-1}_*D$
  is also an snc divisor where $D$ is the divisorial part of the discriminant locus
  of $f\circ g$. This can be constructed the following way: let
  $g_0:(Y_0,\Gamma_0)\to (X,\Delta)$ be a (good) resolution of $(X,\Delta)$ and let
  $D_0\subseteq B$ denote the divisorial part of the discriminant locus of
  $f_0\leteq f\circ g_0$, i.e., the smallest effective reduced divisor
  $D_0\subseteq B$ such that
  $g_0\resto{Y_0\setminus f_0^{-1}D_0}: Y_0\setminus f_0^{-1}D_0\to B\setminus D_0$
  is smooth in codimension $1$. Now let
  $g_1:(Y,\Gamma+g_1^{-1}f_0^{-1}D_0)\to (Y_0,\Gamma_0+f_0^{-1}D_0)$ be a good
  resolution such that $g_1$ is an isomorphism over $Y_0\setminus f_0^{-1}D_0$ and
  let $g=g_0\circ g_1:(Y,\Gamma)\to (X,\Delta)$.

  Note that if $\Delta=\emptyset$, then we will often drop $\Gamma$ from the notation
  and just say that $g:Y\to X$ is a \emph{good resolution (with respect to $f$)}.

  A \emph{good hyperresolution} of $(X,\Delta)$, denoted by
  $\varepsilon_\kdot:(X_\kdot,\Delta_\kdot)\to (X,\Delta)$ consists of a
  hyperresolution $\varepsilon_\kdot:X_\kdot\to X$ such that for each $i\in\bN$,
  $\dim X_i\leq \dim X-i$ and for
  $\Delta_\kdot\leteq X_\kdot\setminus (X_\kdot\times_X(X\setminus\Delta))$, either
  $\Delta_i$ is an snc divisor on $X_i$, or $\Delta_i=X_i$.
\end{definition}

\subsection{Hyperfiltrations and spectral sequences}

Let $\mcA$ and $\mcB$ be abelian categories and $D(\mcA)$ and $D(\mcB)$ their derived
categories respectively. Let $\Phi:\mcA\to\mcB$ be a left exact additive functor and
assume that $\myR\Phi: D(\mcA)\to D(\mcB)$, the right derived functor of $\Phi$
exists.

\begin{definition}\cite{Kovacs05a}*{1.2.1}
  Let $\sfK\in \ob (D^b(\mcA))$ be a bounded complex. A \emph{bounded
    hyperfiltration} $\bF^\kdot\sfK$ of $\sfK$ consists of a set of objects
  $\bF^ j\sfK\in \ob (D^b(\mcA))$ for $j=l, \dots, k+1$, for some $l,k\in\bZ$ and
  morphisms
  \[
  \varphi_j\in\Hom_{D^b(\mcA)} (\bF^{j+1}K, \bF^ jK)\qquad\text{for }j=l,\dots,k,
  \] 
  where $\bF^ l \sfK \simeq \sfK$ and $\bF^ {k+1}\sfK \simeq 0$. $\bF^ jK$ will be
  denoted by $\bF^ j$ when no confusion is likely. For convenience let
  $\bF^ i\sfK=\sfK$ for $i< l$ and $\bF^ i\sfK=0$ for $i> k$.
  The \emph{$p$-th associated graded complex} of a hyperfiltration $\bF^\kdot\sfK$ is
  \[
  \bG r^ p_{\bF^\kdot\sfK}:= M(\varphi_p) ,
  \]  
  the mapping cone of the morphism $\varphi_p$.
   
  Let $\bF^\kdot\sfA$ be a hyperfiltration of the object $\sfA$ and
  $\Xi:D(\mcA)\to D(\mcB)$ a functor. Then, there is a natural hyperfiltration of
  $\Xi(\sfA)$ given by
  \[
    \bF^j(\Xi(\sfA))\leteq \Xi(\bF^j\sfA)
  \]
  for each $j\in \bZ$.  We will always consider the object $\Xi(\sfA)$ with this
  natural hyperfiltration, unless otherwise specified.

  Next let $\bF^\kdot\sfA$ and $\bF^\kdot\sfB$ be hyperfiltrations of the objects
  $\sfA$ and $\sfB$ of $D(\mcA)$ respectively. Then, a \emph{(hyper)filtered
    morphism}\footnote{Strictly speaking these morphisms should be called
    \emph{hyperfiltered}, but for simplicity we will call them \emph{filtered}.}
  between $\sfA$ and $\sfB$ is a collection of compatible morphisms
  $\bF^j\alpha:\bF^j\sfA\to\bF^j\sfB$, i.e. for each $j\in \bZ$, the diagrams
  \[
    \xymatrix@C=4em{%
      \bF^{j+1}\sfA \ar[r]^{\bF^{j+1}\alpha} \ar[d] & \bF^{j+1}\sfB \ar[d] \\
      \bF^j\sfA \ar[r]^{\bF^{j}\alpha} & \bF^{j}\sfB \\
    }
  \]
  are commutative in $D(\mcA)$. Notice that in this case these
  morphisms induce a morphism
  $\alpha^j:\bG r^j_{\bF^\kdot\sfA}\to \bG r^j_{\bF^\kdot\sfB}$, for each $j\in \bZ$.

  A filtered morphism $\alpha:\sfA\to\sfB$ is a \emph{filtered quasi-isomorphism} if
  the induced morphism
  $\alpha^j:\bG r^j_{\bF^\kdot\sfA}\overset\simeq\longrightarrow \bG
  r^j_{\bF^\kdot\sfB}$ is an isomorphism for each $j\in \bZ$. It is easy to see, and
  left to the reader, that a filtered quasi-isomorphism (of bounded complexes) is
  necessarily a quasi-isomorphism.
\end{definition}

\begin{example}
  Let $\sfA\in C(\mcA)$ be a complex of objects of the abelian category $\mcA$ and
  let $\sfA =\colon F^0\supseteq F^1 \supseteq \dots \supseteq F^r=0$ be a
  filtration of $\sfA$. Considering the induced morphisms $F^{j+1}\to F^j$ in
  $D^b(\mcA)$ defines a hyperfiltration of the object $\sfA$.
\end{example}

\section{Relative Du Bois complexes of $p$-forms}\label{sec:rel-db-cxs}

\newcommand\bbeta{a_p}%
\newcommand\felta{f}%
\newcommand\ffelta{{f'}}%
\newcommand\gelta{{fg}}%


\noin Our aim in this section is to construct, for all flat morphisms to regular base
schemes, an analogue of relative logarithmic $p$-forms for morphisms of snc pairs.  To
do so, following the construction in \cite{Kovacs05a} (reviewed in
\autoref{sec:appendix}), we will work in the derived category $D^b(X)$.  We use the
notation $\underline\Omega^p_{X/B}(\log \Delta/D)$ to denote this object for a
morphism of pairs $f:(X,\Delta)\to (B,D)$. The ``$D$" is included in the notation to
emphasize the fact that the construction depends on $D$ as well.

With hyperfiltrations playing a role here, similar to that of filtrations in an
abelian category, our first goal is to use these objects to construct a functorial
filtration of $\underline \Omega^p_{X}(\log \Delta)$ (\autoref{thm:log-DB}).  Our
next goal is to establish a connection between $\underline \Omega^p_X$ and
$\underline \Omega^p_X(\log \Delta)$, as hyperfiltered objects
(\autoref{thm:pullback}).  This is where the notion of discrepancy (as was mentioned
in the introduction) naturally appears.  Our final goal in this section is to extend
these relations to distinguished triangles arising from such hyperfiltrations
(\autoref{cor:pullback}).  The latter is of particular interest in the context of
VHSs, as we will see in \autoref{sect:Section3-DBExtension}.

\noindent
We will use the terminology, notation and conventions developed in
\autoref{sec:definitions-notation}.

\begin{definition}
  \label{def:cx-db-forms}
  Let $(X,\Delta)$ be a reduced pair (\autoref{def:families-pairs}) and
  $\varepsilon_\kdot:(X_\kdot,\Delta_\kdot)\to (X,\Delta)$ a {good hyperresolution}
  (\autoref{def:good-hyper-res}).  The \emph{logarithmic Deligne-Du~Bois complex} (or
  \emph{logarithmic DB complex} for short) of $(X,\Delta)$ is defined as
  $\underline\Omega_X^\kdot(\log \Delta)\leteq \myR\left(\varepsilon_\kdot\right)_*
  \Omega_{X_\kdot}^\kdot(\log \Delta_\kdot)$. This is an object in the bounded
  filtered derived category of coherent sheaves on $X$, and the corresponding
  filtration (induced by the filtration b\^ete on each component of $X_{\kdot}$) is
  denoted by $F^\kdot_{DB}\leteq F^\kdot_{DB}\underline\Omega_X^\kdot(\log \Delta)$.
  Both the obeject and this filtration is independent from the good hyperresolution
  used in the definition.  The associated graded objects of this filtration give rise
  to the \emph{complexes of logarithmic DB $p$-forms}:
  $\underline\Omega_X^p(\log \Delta)\leteq
  \left(\bGr^p_{F^\kdot_{DB}}\underline\Omega_X^\kdot(\log \Delta)\right)[p]$. The
  reader is referred to \cite{GNPP88}*{IV.2.1} for details on this definition and
  basic properties of these complexes.
\end{definition}

\noin We will construct relative versions of these complexes, but first we need a
notation.

\begin{definition}\label{def:log-DB-filtration-setup}
  Let $f: (X, \Delta) \to (B,D)$ be a morphism of snc pairs,
  $\Phi_X\leteq f^*\Omega^1_B(\log D)$, $\Psi_X\leteq \Omega^1_X(\log \Delta)$, and
  $\theta_X:\Phi_X\to\Psi_X$ the natural morphism induced by $f$.  Using the notation
  of \autoref{sec:appendix} (cf.~\cite{Kovacs05a}), set
  $\bF^\kdot_\felta\underline\Omega^q_{X/B} (\log\Delta)\leteq
  \bF^\kdot\!\!\bigwedge^p\!\Psi_X$ and define
  $\underline\Omega^p_{X/B} (\log\Delta/D)\leteq \underline\frQ_{\theta_X}^p$, where
  $\underline\frQ_{\theta_X}^p$ is the object constructed in \autoref{spseqegy}
  (cf.~\cite{Kovacs05a}*{2.7}).

  Next, let $f: (X, \Delta) \to (B,D)$ be a morphism of pairs and assume that $(B,D)$
  is an snc pair. I.e., do not assume that $(X,\Delta)$ is snc.  Let
  $\varepsilon_\kdot:(X_\kdot,\Delta_\kdot)\to (X,\Delta)$ be a {good
    hyperresolution}. Then, as in \autoref{def:cx-db-forms}, the logarithmic
  Deligne-Du~Bois complex of $(X,\Delta)$ is defined as
  $\underline\Omega_X^\kdot(\log \Delta)\leteq \myR\left(\varepsilon_\kdot\right)_*
  \Omega_{X_\kdot}^\kdot(\log \Delta_\kdot)$.  Using this representative define a
  filtration as follows: Let $n=\dim X$, $d=\dim B$, and for each
  $0\leq p,q \leq \dim (X/B)=n-d$, and $0\leq j\leq \dim B$, let
  \begin{equation}
    \label{eq:1}
    \bF_\felta^j\underline\Omega_X^q(\log \Delta)\leteq \myR(\varepsilon_\kdot)_*
    \bF^j\Omega_{X_\kdot}^q(\log \Delta_\kdot),
  \end{equation}
  and
  \begin{equation}
    \label{eq:2}
    \underline\Omega_{X/B}^p(\log \Delta/D)\leteq \myR(\varepsilon_\kdot)_*
    \underline\Omega_{X_\kdot/B}^p(\log \Delta_\kdot/D).
  \end{equation}
  The object $\underline \Omega^p_{X/B} (\log \Delta/D) \in D^b(X)$ will be called
  the \emph{$p^\text{th}$-relative logarithmic Deligne-Du~Bois complex of
    $f: (X, \Delta) \to (B,D)$} or simply the \emph{complex of relative logarithmic
    DB $p$-forms of $f$}.

  Next, we will prove that these objects are well-defined and satisfy a list of
  useful properties.
\end{definition}

\begin{theorem}\label{thm:log-DB}%
  Let $f: (X, \Delta) \to (B,D)$ be a 
  morphism of pairs and assume that $(B,D)$ is an snc pair.  Let $n=\dim X$ and
  $d=\dim B$. Then, for each $0\leq p,q \leq \dim (X/B)=n-d$, the objects
  $\underline\Omega^p_{X/B}(\log\Delta/D) \in \Ob D^b(X)$ and a
  the hyperfiltration $\bF_\felta^\kdot \underline\Omega^q_{X}(\log \Delta)$ satisfy
  the following properties.

  \begin{enumerate}
  \item\label{item:18} The object
    $\underline\Omega^p_{X/B}(\log\Delta/D) \in \Ob D^b(X)$ is independent from the
    good hyperresolution used in its definition. In other words, any two objects
    defined as in \autoref{def:log-DB-filtration-setup} using possibly different good
    hyperresolutions are isomorphic in $D^b(X)$.
  \item\label{item:1}
    $\bF_\felta^0\underline\Omega^q_{X}(\log \Delta)=\underline\Omega^q_{X}(\log
    \Delta)$, and $\bF_\felta^{d+1}\underline\Omega^q_{X}(\log \Delta)=0$.
  \item\label{item:8} Let $\phi:(\wt X,\wt \Delta)\to (X,\Delta)$ be a log
    resolution. Then
    \[
      \underline\Omega^{n-d}_{X/B}(\log\Delta/D)\simeq \myR\phi_*\omega_{\wt
        X/B}(\wt\Delta-(f\phi)^{*}D).
    \]
  \item\label{item:2} For each $0\leq j\leq d$,
    \[
      \bG r^j_\felta\underline\Omega^q_{X}(\log \Delta) \leteq \bG
      r^j_{\bF^\kdot_\felta}\underline\Omega^q_{X}(\log \Delta) \simeq f^*\Omega^{j}_B (\log D)  
      \otimes  \underline \Omega^{q-j}_{X/B}(\log\Delta/D)  .
    \]
  \item\label{item:3} The hyperfiltration
    $\bF_\felta^\kdot \underline\Omega^q_{X}(\log \Delta)$ is functorial in the
    following sense. Let $g:(Y,\Gamma)\to (X,\Delta)$ be a morphism of pairs such
    that $\dim Y=\dim X=n$.  Then, for each $0\leq q \leq n-d$, there exists a natural
    filtered morphism in $D^b(X)$
    \[
      \bF_\felta^\kdot\underline\Omega^q_{X}(\log \Delta) \longrightarrow \myR g_*
      \bF_\gelta^\kdot\underline\Omega^q_{Y}(\log \Gamma).
    \]
  \item\label{item:4} The formation of $\underline\Omega^p_{X/B}(\log\Delta/D)$ is
    functorial in the following sense. Let $g:(Y,\Gamma)\to (X,\Delta)$ be a morphism
    of pairs such that $\dim Y=\dim X=n$.
    Then, for each $0\leq p \leq n-d$, there exists a natural 
    morphism in $D^b(X)$,
    \[
      \underline\Omega^p_{X/B}(\log\Delta/D) \longrightarrow \myR g_*
      \underline\Omega^p_{Y/B}(\log \Gamma/D).
    \]
  \item\label{item:5} If $f: (X, \Delta)\to (B,D)$ is an snc morphism, then
    there is a natural filtered quasi-isomorphism
    \[
      \xymatrix@C2em{%
        \bF_\felta^\kdot\underline\Omega^p_{X} (\log\Delta) 
        \overset\simeq\longrightarrow
        F_{K}^\kdot\Omega^p_{X}(\log\Delta),}
    \]
    where $F_{K}^\kdot$ is the Koszul filtration
    \emph{(\autoref{def:Koszul-filtration})}.
  \end{enumerate}
\end{theorem}

\begin{remark}
  When $\Delta$ and $D$ are empty, we will suppress the ``$\log$'' term from the
  notation. In particular, we will use the notation
  \begin{gather*}
    \underline\Omega^p_{X} \leteq \underline\Omega^p_{X} (\log\emptyset) \\
    \bF^\kdot_{X/B}\underline\Omega^p_{X}\leteq \bF_\felta^\kdot
    \underline\Omega^q_{X}(\log \emptyset), \text{ and} \\
    \underline\Omega^p_{X/B} \leteq \underline\Omega^p_{X/B}
    (\log\emptyset/\emptyset).
  \end{gather*}
\end{remark}



\begin{notation}\label{not:not-for-Gr}
  To avoid cumbersome notation, as in \autoref{thm:log-DB}\autoref{item:2}, we will
  use $\bG r^p_f$ to denote $\bG r^p_{\bF_f^{\kdot}}$, where $\bF^{\kdot}_f$ is the
  hyperfiltration $\bF^{\kdot}_f \underline \Omega^q(\log \Delta)$ in
  \autoref{thm:log-DB}.
\end{notation}

\begin{proof}[Proof of \autoref{thm:log-DB}]
  %
  First, assume in addition that $(X,\Delta)$ is also an snc pair.  Then the
  statements \autoref{item:1} and \autoref{item:2} follow from \autoref{spseqegy},
  and \autoref{item:3} follows from \cite{Kovacs05a}*{4.1}.  For \autoref{item:8},
  first observe that both sides are independent of the choice of $\phi$. This follows
  from \cite{DuBois81}*{6.3} for the left hand side and from \cite{MR2918171}*{2.10}
  (cf.~\cite{KollarSingsOfTheMMP}*{10.34}) for the right hand side. In particular, in
  the snc case we may use $\phi=\id$, and in that case \autoref{item:8} follows from
  \autoref{def:Q}.
  
  Next, let $(X,\Delta)$ be arbitrary and let
  $\varepsilon_\kdot:(X_\kdot,\Delta_\kdot)\to (X,\Delta)$ be a {good
    hyperresolution}.  Using \autoref{def:log-DB-filtration-setup}, \autoref{item:1},
  \autoref{item:8}, \autoref{item:2}, and \autoref{item:3} follow from the snc case
  above: \autoref{item:1} follows directly, \autoref{item:8} follows from the snc
  case, the definition of a good hyperresolution, \autoref{def:good-hyper-res}, and
  \autoref{eq:2}. Item \autoref{item:3} follows by the functoriality of the snc
  case. For \autoref{item:2}, further note that as $(B,D)$ is an snc pair,
  $f^*\Omega_B^{q-j}(\log D)$ is locally free, so one can use the projection formula.

  Statement \autoref{item:4} follows by a descending induction on $p$. The induction
  can be started by \autoref{item:1} and the inductive step follows from
  \autoref{item:2} and \autoref{item:3}. Indeed, choose a good hyperresolution
  $\mu_\kdot:(Y_\kdot,\Gamma_\kdot)\to (Y,\Gamma)$, which is compatible with the
  chosen good hyperresolution of $(X,\Delta)$, i.e., there is a commutative
  diagram
  \[
    \xymatrix{%
      (Y_\kdot,\Gamma_\kdot) \ar[d]_{\mu_\kdot} \ar[r]^-{g_\kdot} &
      (X_\kdot,\Delta_\kdot)      \ar[d]^{\varepsilon_\kdot}  \\
      (Y,\Gamma) \ar[r]_g & (X,\Delta).  }
  \]
  Then, the following diagram is commutative by \autoref{item:3}:
  \begin{equation}
    \label{eq:12}
    \begin{aligned}
      \xymatrix{%
        \bF^{j+1}_\felta\underline\Omega^p_{X} (\log \Delta) \ar[d] \ar[r] &
        \bF^j_\felta\underline\Omega^p_{X} (\log \Delta) \ar[d] \ar[r] &
        \bG r^j_\felta\underline\Omega^p_{X} (\log \Delta) \ar[r]^-{+1} & \\
        \myR g_* \bF^{j+1}_\gelta\underline\Omega^p_{Y} (\log \Gamma) \ar[r] & \myR
        g_*\bF^j_\gelta\underline\Omega^p_{Y} (\log \Gamma) \ar[r] & \myR g_*\bG
        r^j_\gelta\underline\Omega^p_{Y} (\log \Gamma) \ar[r]^-{+1} &, }
    \end{aligned}
  \end{equation}
  and hence induces a compatible natural morphism
  $\bG r^j_\felta\underline\Omega^p_{X} (\log \Delta)\to \myR g_*\bG
  r^j_\gelta\underline\Omega^p_{Y} (\log \Gamma)$.  The already proven
  \autoref{item:2} implies that
  $\underline\Omega^p_{X/B}(\log\Delta/D)\simeq \bG r^0_\felta\underline\Omega^p_{X}
  (\log \Delta)$ and
  $\underline\Omega^p_{Y/B}(\log \Gamma/D)\simeq \bG r^0_\felta\underline\Omega^p_{Y}
  (\log \Gamma)$. This finishes the proof of \autoref{item:4}, and then
  \autoref{item:5} follows from the construction of $\underline\frQ_{\theta_X}^p$
  carried out in this case (cf.~\autoref{sec:appendix}, especially \autoref{def:Q},
  and \cite{Kovacs05a}*{\S 2}). The main point is that the cokernel of the morphism
  $f^*\Omega^1_B(\log D)\to\Omega^1_X(\log\Delta)$ is locally free and hence its
  exterior powers satisfy the required properties, cf.~\autoref{easy}. In fact, the
  construction outlined in \autoref{sec:appendix} was modeled after this case.

  Finally, to prove \autoref{item:18}, observe that we have just proved that the
  other properties stated in the theorem hold for the corresponding object defined by
  any good hyperresolution of $(X,\Delta)$. For any two good hyperresolution there
  exists a third that maps to and is compatible with both of the others, so it is
  enough to prove \autoref{item:18} for two such good hyperresolutions. Then the
  proofs of \autoref{item:3} and \autoref{item:4} show that there is a natural
  filtered morphism between the two objects defined by the two good
  hyperresolutions. It follows from \autoref{item:8} that the induced morphism is an
  isomorphism for $p=n-d$ and then descending induction using the commutative diagram
  \autoref{eq:12} shows that \autoref{item:18} holds for all $p$. This finishes the
  proof of all the claims in the theorem.
\end{proof}

\noin Next we will compare these objects obtained with respect to different bases
replacing $(B,D)$.  We will be using the standard
$\omega_{B/B'}\leteq \omega_B\otimes \tau^*\omega_{B'}^{-1}$ notation.

\begin{theorem}\label{thm:diff-bases}
  Using the notation from \autoref{thm:log-DB}, in addition let
  $\tau:(B,D)\to (B',D')$ be another morphism of pairs, such that $(B',D')$ is also
  an snc pair and $\tau$ is a dominant generically finite morphism.  Let
  $f'=\tau\circ f$ and $\Gamma\leteq D-\tau^{*}D'$.
  %
  Then for each $0\leq p, q\leq \dim (X/B)=n-d$, and $0\leq j\leq d=\dim B=\dim B'$,
  \begin{enumerate}
  \item\label{item:9}  there exists a natural
    morphism
    \[
      \mu_p: \underline\Omega^p_{X/B'} (\log\Delta/D') \longrightarrow
      \underline\Omega^p_{X/B} (\log\Delta/D) \otimes \left(
        f^*\omega_{B/B'}(\Gamma)\right)^{n-d-p+1}.
    \]
  \item\label{item:10} there exists a natural morphism
    \[
      \nu_{j,q}: \bF_\ffelta^j\underline\Omega^q_{X} (\log\Delta) \longrightarrow
      \bF_\felta^j\underline\Omega^q_{X} (\log\Delta) \otimes \left(
        f^*\omega_{B/B'}(\Gamma)\right)^{n-d-q+j}\text{, and}
    \]
  \item\label{item:11} the natural morphisms in \autoref{item:9} and
    \autoref{item:10} are compatible in the following sense. For each $q$ and $j$
    there exists a commutative diagram of distinguished triangles,
    \[
      \xymatrix{%
        \bF^{j+1}_\ffelta\underline\Omega^q_{X} (\log \Delta) \ar[d]^{\nu_{j+1,q}}
        \ar[r] & \bF^j_\ffelta\underline\Omega^q_{X} (\log \Delta)
        \ar[d]^{\nu_{j,q}\otimes \varsigma} \ar[r] & {\begin{aligned}
            \underline\Omega^{q-j}_{X/B}(\log\Delta/D') \\ \otimes f^*\Omega^{j}_{B'}
            (\log D')\end{aligned}}
        \ar[d]^{\mu_{q-j}\otimes\wedge^j\varrho} \ar[r]^-{+1} & \\
        \bF^{j+1}_\felta\underline\Omega^q_{X} (\log \Delta)\otimes \sL_{j,q} \ar[r]
        & \bF^j_\felta\underline\Omega^q_{X} (\log \Delta)\otimes \sL_{j,q} \ar[r] &
        {\begin{aligned} \underline\Omega^{q-j}_{X/B}(\log\Delta/D) \\ \otimes
            \sL_{j,q} \otimes f^*\Omega^{j}_B (\log D) \end{aligned}}
        \ar[r]^-{+1} &, }
    \]
    where $\sL_{j,q} =\left(f^*\omega_{B/B'}(\Gamma)\right)^{n-d-q+j+1}$ and
    \[
      \varsigma=\varrho\otimes \id_{(\det \Phi'_X)^{-1}} : \sO_X \to
      f^*\omega_{B/B'}(\Gamma).
    \]
  \item\label{item:12} there exists a natural filtered morphism
    \[
      \nu_q: \bF_\ffelta^\kdot\underline\Omega^q_{X} (\log\Delta) \longrightarrow
      \bF_\felta^\kdot\underline\Omega^q_{X} (\log\Delta) \otimes
      \left(f^*\omega_{B/B'}(\Gamma)\right)^{n-q}.
    \]
  \end{enumerate}
\end{theorem}


\begin{proof}
  We will use the notation of \autoref{sec:appendix}. In particular, let
  $\Phi_X, \Phi'_X$ and $\Psi_X$ be locally free sheaves on $X$ of rank $k,k'$ and
  $n$ respectively, and let $\varrho: \Phi_X'\to \Phi_X$ and
  $\theta_X:\Phi_X\to\Psi_X$ be two morphisms. Further let
  $\theta'_X\leteq \theta_X\circ\varrho: \Phi_X'\to \Psi_X$.

  Then $\varrho$ induces a natural map between the filtration diagrams corresponding
  to the morphisms $\theta_X',\theta_X$ (the maps go from the ones associated to
  $\theta'_X$ to those associated to $\theta_X$ induced by the morphisms
  $\wedge^r\varrho: \bigwedge^r \Phi_X'\to \bigwedge^r\Phi_X$ for various $r$).

  In particular, let $\Phi_X\leteq f^*\Omega^1_B(\log D)$ and
  $\Phi_X'\leteq (f')^*\Omega^1_{B'}(\log D')$. Then for the objects $\frF^p_i$
  defined in \ref{indstep} one obtains a natural morphism
  \begin{equation}
    \label{eq:4}
    \mu_p^\circ: \frF^{n-d-p}_{n-d-p}(\theta'_X) \longrightarrow
    \frF^{n-d-p}_{n-d-p}(\theta_X) , 
  \end{equation}
  for each $0\leq p \leq n-d=\dim(X/B)$.
  Using 
  the definition,
  \begin{gather}\label{eq:5}
    \frQ_{\theta_X}^p\leteq \frF^{n-d-p}_{n-d-p}(\theta_X)\otimes
    \left(\det\Phi_X\right)^{-(n-d-p+1)}, \\
    \label{eq:6}
    \frQ_{\theta'_X}^p\leteq \frF^{n-d-p}_{n-d-p}(\theta'_X)\otimes
    \left(\det\Phi'_X\right)^{-(n-d-p+1)},
  \end{gather}
   (cf.~\autoref{def:Q}), 
  and the fact that in this case we have
  \begin{equation*}
    \det\Phi_X\simeq f^*\Omega^d_B(\log D) \qquad\text{and}\qquad
    \det\Phi'_X\simeq (f')^*\Omega^{d}_{B'}(\log D')
  \end{equation*}
  we find that there exist natural morphisms
  \begin{multline*}
    \mu_p\leteq \mu_p^\circ\otimes
    \id_{\left(f^*\omega_{B/B'}(\Gamma)\right)^{-(n-d-p+1)}} :
    \underline\Omega^p_{X/B'}
    (\log\Delta/D')\simeq \underline\frQ_{\theta'_X}^p \longrightarrow \\
    \longrightarrow\underline\frQ_{\theta_X}^p \otimes
    \left(f^*\omega_{B/B'}(\Gamma)\right)^{n-d-p+1}\simeq 
    \underline\Omega^p_{X/B} (\log\Delta/D) \otimes
    \left(f^*\omega_{B/B'}(\Gamma)\right)^{n-d-p+1}.
  \end{multline*}
  This proves \autoref{item:9}. The same argument, used for
  \begin{equation*}
    \nu_{j,q}^\circ: \frF^{n-q}_{n-d-q+j}(\theta'_X) \longrightarrow
    \frF^{n-q}_{n-d-q+j}(\theta_X) 
  \end{equation*}
  instead of the morphism in \autoref{eq:4} and using the definition
  
  \begin{equation}
    \label{eq:7}
    \bF^j_\felta\underline\Omega^q_{X} (\log \Delta)
    \leteq \frF^{n-q}_{n-d-q+j}(\theta_X)\otimes
    \left(\det\Phi_X\right)^{-(n-d-q+j)}, \\
  \end{equation}
  (cf.~\autoref{def:Q}),
  gives
  \begin{multline*}
    \nu_{j,q}= \nu_{j,q}^\circ \otimes
    \id_{\left(f^*\omega_{B/B'}(\Gamma)\right)^{-(n-d-q+j)}} :
    \bF_\ffelta^j\underline\Omega^q_{X} (\log\Delta) \longrightarrow \\
    \longrightarrow \bF_\felta^j\underline\Omega^q_{X} (\log\Delta) \otimes
    \left(f^*\omega_{B/B'}(\Gamma)\right)^{n-d-q+j},
  \end{multline*}
  which proves \autoref{item:10}. In fact, following this argument one is led to
  consider distinguished triangles of the form
  \[
    \xymatrix@C4ex{%
      \frF^{n-q}_{n-d-q+j+1}(\theta'_X) \ar[d]^{\nu_{j+1,q}^\circ} \ar[r] &
      \frF^{n-q}_{n-d-q+j}(\theta'_X) \otimes \det\Phi'_X
      \ar[d]^{\nu_{j,q}^\circ\otimes \det\varrho} \ar[r] &
      \frF^{n-d-q+j}_{n-d-q+j}(\theta'_X) \otimes \bigwedge^j \Phi'_X
      \ar[d]^{\mu_{q-j}^\circ\otimes\wedge^j\varrho} \ar[r]^-{+1}
      & \\
      \frF^{n-q}_{n-d-q+j+1}(\theta_X) \ar[r] & \frF^{n-q}_{n-d-q+j}(\theta_X) \otimes
      \det\Phi_X \ar[r] & \frF^{n-d-q+j}_{n-d-q+j}(\theta_X) \otimes \bigwedge^j \Phi_X
      \ar[r]^-{+1} &, \\}
  \]
  where the vertical arrows are induced by the morphism $\varrho: \Phi_X'\to \Phi_X$
  as indicated. Tensoring this diagram with $\left(\det\Phi'_X\right)^{-(n-d-q+j+1)}$
  gives
  \begin{equation*}
    \xymatrix{%
      \bF^{j+1}_\ffelta\underline\Omega^q_{X} (\log \Delta)
      \ar[d]^{\nu_{j+1,q}} \ar[r] &
      \bF^j_\ffelta\underline\Omega^q_{X} (\log \Delta) \ar[d]^{\nu_{j,q}\otimes
        \varsigma} \ar[r] & 
       {\begin{aligned}     \underline\Omega^{q-j}_{X/B}(\log\Delta/D') \\   \otimes f^*\Omega^{j}_{B'} (\log D')  \end{aligned}} 
      \ar[d]^{\mu_{q-j}\otimes\wedge^j\varrho} \ar[r]^-{+1} & \\
      \bF^{j+1}_\felta\underline\Omega^q_{X} (\log \Delta)\otimes \sL _{j,q}\ar[r] &
      \bF^j_\felta\underline\Omega^q_{X} (\log \Delta)\otimes \sL _{j,q}\ar[r] &
      {\begin{aligned}   \underline\Omega^{q-j}_{X/B}(\log\Delta/D)  \\   \otimes \sL _{j,q}
      \otimes f^*\Omega^{j}_B (\log D) \end{aligned}}
      \ar[r]^-{+1} &,}
  \end{equation*}
  where $\sL_{j,q} =\left(f^*\omega_{B/B'}(\Gamma)\right)^{n-d-q+j+1}$ and
  $\varsigma=\wedge^d\varrho\otimes \id_{(\det \Phi'_X)^{-1}} : \sO_X \to
  f^*\omega_{B/B'}(\Gamma)$. This proves \autoref{item:11}.

  Finally, in order to prove \autoref{item:12}, first recall that
  $\bF^{j+1}_\felta\underline\Omega^q_{X} (\log \Delta)=0$ and
  $\bF^{j+1}_\ffelta\underline\Omega^q_{X} (\log \Delta)=0$ by
  \autoref{thm:log-DB}\autoref{item:1}, so we may assume that $j<\dim B=d$. With that
  restriction, consider the morphism
  \begin{equation}
    \label{eq:9}
    \varsigma^{d-j-1} : \sL_{j,q} =
    \left(f^*\omega_{B/B'}(\Gamma)\right)^{n-d-q+j+1}\to
    \left(f^*\omega_{B/B'}(\Gamma)\right)^{n-q} =\colon \sL_q. 
  \end{equation}
  This morphism allows us to add one more row to the above commutative diagram:
  \begin{equation*}
    \xymatrix{%
      \bF^{j+1}_\ffelta\underline\Omega^q_{X} (\log \Delta)
      \ar[d]^{\nu_{j+1,q}} \ar[r] &
      \bF^j_\ffelta\underline\Omega^q_{X} (\log \Delta) \ar[d]^{\nu_{j,q}\otimes
        \varsigma} \ar[r] & 
       {\begin{aligned}     \underline\Omega^{q-j}_{X/B}(\log\Delta/D') \\ 
          \otimes f^*\Omega^{j}_{B'} (\log D')  \end{aligned}}    
      \ar[d]^{\mu_{q-j}\otimes\wedge^j\varrho} \ar[r]^-{+1} & \\
      \bF^{j+1}_\felta\underline\Omega^q_{X} (\log \Delta)\otimes \sL _{j,q}\ar[r]
      \ar[d]^{\id\otimes \varsigma^{d-j-1}} 
      &
      \bF^j_\felta\underline\Omega^q_{X} (\log \Delta)\otimes \sL _{j,q}\ar[r]
      \ar[d]^{\id\otimes \varsigma^{d-j-1}}
      &
      {\begin{aligned} \underline\Omega^{q-j}_{X/B}(\log\Delta/D) \\  
      \otimes \sL _{j,q} \otimes f^*\Omega^{j}_B (\log D)  \end{aligned}}
      \ar[d]^{\id\otimes \varsigma^{d-j-1}} \ar[r]^-{+1} &,
      \\
      \bF^{j+1}_\felta\underline\Omega^q_{X} (\log \Delta)\otimes \sL _{q}\ar[r] &
      \bF^j_\felta\underline\Omega^q_{X} (\log \Delta)\otimes \sL _{q}\ar[r] &
      {\begin{aligned}        \underline\Omega^{q-j}_{X/B}(\log\Delta/D) \\ 
       \otimes \sL _{q} \otimes f^*\Omega^{j}_B (\log D) \end{aligned}}
      \ar[r]^-{+1} &.
    }
  \end{equation*}
  Now, erasing the middle row gives us commutative diagrams, for each $j$, with the
  same line bundle multiplier in the last row. In other words this shows that there
  exists a natural filtered morphism
  \[
    \nu_q: \bF_\ffelta^\kdot\underline\Omega^q_{X} (\log\Delta) \longrightarrow
    \bF_\felta^\kdot\underline\Omega^q_{X} (\log\Delta) \otimes \sL_q,
  \]
  as claimed in \autoref{item:12} (cf.~\autoref{eq:9}).
\end{proof}

\begin{theorem}\label{thm:pullback}
  Using the notation from \autoref{thm:log-DB} and \autoref{thm:diff-bases}, we have
  that for each $0\leq p \leq \dim(X/B)$ there exists an integer
  $0 \leq \bbeta \leq \dim X-p$,
  \begin{enumerate}
  \item\label{item:13} a natural filtered morphism
    \begin{equation*}
      \bF^\kdot_{X/B}\underline\Omega^p_{X}
      \longrightarrow   \underline      \bF_\felta^\kdot \underline
      \Omega^p_{X}(\log\Delta) 
      \otimes 
      \sO_X(\bbeta \tightdot f^*D ), \text{ and}
    \end{equation*}
  \item\label{item:6} a natural morphism
    \begin{equation*}
      \underline\Omega^p_{X/B}
      \longrightarrow   \underline       \Omega^p_{X/B}(\log\Delta/D)   \otimes
      \sO_X(\bbeta \tightdot f^*D ).
    \end{equation*}
  \end{enumerate}
  In addition, let $g:(Y,\Gamma)\to (X,\Delta)$ be another morphism of pairs such
  that $\dim Y=\dim X=n$. Then there exist
  \begin{enumerate}[resume]
  \item\label{item:14} a natural filtered morphism
    \begin{equation*}
      \bF^\kdot_{X/B}\underline\Omega^p_{X}
      \longrightarrow      \myR g_*   \bF_{fg}^\kdot \underline
      \Omega^p_{Y}(\log\Gamma) 
      \otimes 
      \sO_X(\bbeta \tightdot f^*D ), \text{ and}
    \end{equation*}
  \item\label{Map1}\label{item:7}  a natural morphism
    \begin{equation*}
      \varkappa_p: \underline\Omega^p_{X/B}
      \longrightarrow   \myR g_* \underline \Omega^p_{Y/B}(\log \Gamma/D)   \otimes
      \sO_X(\bbeta \tightdot f^*D ).  
    \end{equation*}
  \end{enumerate}
\end{theorem}

\begin{proof}
  Let $D'\leteq \emptyset$ and consider the morphism of pairs $\tau:(B,D)\to
  (B,D')$. Observe that $\omega_{B/B'}\simeq \sO_B$ and $\Gamma=D$, hence
  \autoref{item:13} follows from \autoref{thm:diff-bases}\autoref{item:12}.  Then
  \autoref{item:6} follows from \autoref{item:13} and
  \autoref{thm:log-DB}\autoref{item:2}.  The required natural filtered morphisms in
  \autoref{item:14} is simply the composition of the natural filtered morphisms in
  \autoref{item:13} and \autoref{thm:log-DB}\autoref{item:3} (more precisely, the
  latter is twisted with the line bundle $\sO_X(\bbeta \tightdot f^*D)$). Finally,
  \autoref{item:7} follows from \autoref{item:6} and
  \autoref{thm:log-DB}\autoref{item:4}.
\end{proof}

\begin{definition}
\label{def:discrep}
Let $D\subset B$ be a reduced, effective divisor. The smallest non-negative integer
$a\in\bN$ for which a morphism as in \autoref{thm:pullback}\autoref{item:6} for each
$0\leq p\leq \dim(X/B)$ with the choice of $a_p =a \leq \dim X$ exists will be called
the \emph{discrepancy} of $D$ with respect to $f: X\to B$ and will be denoted by
$a_f(D)$.
%
\end{definition}

\setcounter{subsection}{\value{theorem}}
\renewcommand\thesubsection{\thesection.\arabic{subsection}}
\subsection{Koszul triangles}\label{subsect:Koszul}
\renewcommand\thesubsection{\thesection.\Alph{subsection}}
\addtocounter{theorem}1
Let $f: (X, \Delta) \to (B,D)$ be a morphism of pairs and assume that $(B,D)$ is an
snc pair.  Let $n=\dim X$ and $d=\dim B$ and $0\leq p \leq \dim (X/B)=n-d$.

Let $\bG^{0,2}_\felta$ denote the mapping cone of the morphism
$\bF_\felta^2\underline\Omega^p_{X} (\log\Delta) \to
\bF_\felta^0\underline\Omega^p_{X} (\log\Delta)$ and consider the commutative diagram
of distinguished triangles,
\begin{equation}
  \label{eq:11}
  \begin{gathered}
    \xymatrix{%
      \bF_\felta^2\underline\Omega^p_{X} (\log\Delta) \ar[r] \ar[d] &
      \bF_\felta^2\underline\Omega^p_{X} (\log\Delta)
      \ar[r]  \ar[d]  &  0  \ar[d]^\chi  \ar[r]^-{+1}  &  \\
      \bF_\felta^1\underline\Omega^p_{X} (\log\Delta) \ar[r] \ar[d] &
      \bF_\felta^0\underline\Omega^p_{X} (\log\Delta) \ar[r] \ar[d] & \bG
      r^0_\felta\underline\Omega^p_{X} (\log\Delta) \ar[d]
      \ar[r]^-{+1} & \\
      \bG r^1_\felta\underline\Omega^p_{X} (\log\Delta) \ar@{..>}[r] \ar[d]^-{+1} &
      \bG r^{0,2}_\felta \ar@{..>}[r] \ar[d]^-{+1} & \bG
      r^0_\felta\underline\Omega^p_{X}
      (\log\Delta)  \ar[r]^-{+1} \ar[d]^-{+1} & . \\
      &&& }
  \end{gathered}
\end{equation}
Then the dotted arrows exist by
\cite{MR1106349}*{Theorem~1.8}
(cf.~\cite{Kov13}*{Theorem~B.1}) and they maybe identified with the induced morphisms
on the mapping cones.  Therefore we obtain the distinguished triangle
\begin{equation}\label{KTriangle}
  \xymatrix{%
    \bG r^1_\felta\underline\Omega^p_{X} (\log\Delta) \ar[r] &
    \bG r^{0,2}_\felta \ar[r] & \bG r^0_\felta\underline\Omega^p_{X} (\log\Delta)
    \ar[r]^-{+1} &  }
\end{equation}
in which each term is defined as the mapping cone of the vertical morphisms in
\autoref{eq:11} and the morphisms are the ones coming from the mapping cone
construction.  We will refer to this distinguished triangle in \autoref{KTriangle} as
the \emph{$p^\text{th}$-Koszul triangle} of $f: (X, \Delta) \to (B,D)$ and denote it
by $\underline{\Kosz}_{\felta}^{p}(\log \Delta)$.  As before, in case
$\Delta=\emptyset$ and $D=\emptyset$, then we will denote this by
$\underline \Kosz^{p}_{X/B}$.  Replacing
$\bG r^i_\felta\underline\Omega^p_{X} (\log\Delta)$ for $i=0,1$ by isomorphic objects
as in \autoref{thm:log-DB}\autoref{item:2} we obtain an alternative expression for
$\underline{\Kosz}_{\felta}^{p}(\log \Delta)$:
\begin{equation}
  \label{eq:10}
  \xymatrix{%
     f^*\Omega^{1}_B (\log D)\otimes \underline \Omega^{p-1}_{X/B}(\log\Delta/D)
    \ar[r] &    \bG r^{0,2}_\felta \ar[r] &
    \underline \Omega^{p}_{X/B}(\log\Delta/D) 
    \ar[r]^-{+1} & . }
\end{equation}

\begin{remark}\label{rem:Neeman}
  The nine lemma in triangulated categories is somewhat trickier than in abelian
  categories. It is not true that any morphism of triangles induce a distinguished
  triangle on their mapping cones. What \cite{MR1106349}*{Theorem~1.8} and
  \cite{Kov13}*{Theorem~B.1} state is that there exists a $\chi$ (see upper right
  side of \autoref{eq:10}) such that the third row (of mapping cones) forms a
  distinguished triangle. In addition, it follows from \cite{Kov13}*{Theorem~B.1}
  that in the case of \autoref{eq:11} the $\chi$ is in fact uniquely determined. This
  is, of course, not surprising given that the initial object of $\chi$ is $0$, but
  one should remember that we are working in the derived category, so caution is
  warranted.
  
  We would also like to emphasize that we are not merely stating that a distinguished
  triangle exists with the given objects as in \autoref{KTriangle} and
  \autoref{eq:10}, but that the morphisms of the triangle are exactly the ones one
  would hope for, namely the morphisms induced by the mapping cone construction.  In
  particular, this means that the Koszul triangles will inherit any natural property
  carried by the filtrations used in their definition.
\end{remark}

\begin{corollary}\label{cor:pullback}
  Using the above notation, let $g:(Y,\Gamma)\to (X,\Delta)$ be another morphism of
  pairs such that $\dim Y=\dim X=n$. Then the morphisms $\varkappa_p$ obtained in
  \autoref{thm:pullback}\autoref{item:7} are compatible with Koszul triangles, that
  is, there exist natural compatible morphisms of the terms of the following Koszul
  triangles:
  \[
    \varkappa_p : \underline \Kosz^p_{X/B} \longrightarrow \dR g_* \underline
    \Kosz^p_{\gelta} (\log \Gamma)  (a_f(D) \cdot f^* D)
  \]
  (With a slight abuse of notation we will denote these morphisms of Koszul triangles
  by the same symbol).
\end{corollary}


\begin{proof}
  \autoref{thm:pullback} implies that there exist natural morphisms between the terms
  of the diagram \autoref{eq:11} for $f: (X,\emptyset)\to (B,\emptyset)$ and for
  $fg: (Y,\Gamma)\to (B,D)$.  It follows from \autoref{thm:pullback}\autoref{item:13}
  that these morphisms commute with the first two rows and all the columns.
  Then it follows that they also commute with the third row as well, which is exactly
  the desired statement.

  The naturality of $\varkappa_p$ follows from the naturality of the morphisms in
  \autoref{thm:pullback} and the fact that the morphisms in $\underline\Kosz^p_{X/B}$
  are given by the mapping cone construction, as explained in \autoref{rem:Neeman}.
\end{proof}

\begin{remark}\label{rk:refined}
  One can slightly generalize \autoref{cor:pullback} by considering
  \[
    \varkappa_p: \underline \Kosz^p_f (\log \Delta') \longrightarrow \dR g_*
    \underline \Kosz^p_{fg} (\log\Gamma) \big( a_f(D-D') \cdot f^*(D-D') \big),
  \]
  where $0\leq D'\leq D$ and $\Delta'= f^{-1}D' \leq \Delta$.  The proof is the same
  as the one for \autoref{cor:pullback}.
\end{remark}

\begin{not-rem}\label{rem:triangles}
  For an snc morphism, $f: (X, \Delta) \to (B, D)$, we use the notation
  $\Kosz^p_f(\log \Delta)$ to denote the triangle defined by the standard Koszul
  filtration $F^{\kdot}_K$. Note that for such morphisms, using
  \autoref{thm:log-DB}\autoref{item:5}, there is a natural isomorphism of triangles
  $ \underline \Kosz^p_f (\log \Delta) \longrightarrow \Kosz^p_f (\log \Delta)$,
  defined explicitly by
  \[
    \xymatrix{%
      f^* \Omega^1_B (\log D) \otimes \underline\Omega^{p-1}_{X/B} (\log \Delta/D)
      \ar[r] \ar[d]^{\qis} & \bG r^{0,2}_\felta \ar[r] \ar[d] & \underline
      \Omega^p_{X/B} (\log \Delta/D)  \ar[r]^(0.7){+1} \ar[d]^{\qis}  & \\
      f^*\Omega^1_B(\log D) \otimes \Omega^{p-1}_{X/B} (\log \Delta) \ar[r] &
      \Gr^{0,2}_{F_K^{\kdot}} \ar[r] & \Omega^p_{X/B} (\log \Delta) \ar[r]^(0.7){+1}
      & , }
  \]
  where the vertical quasi-isomorphims are the ones defined by
  \autoref{thm:log-DB}\autoref{item:5} (see also \autoref{rem:Neeman}).
 \end{not-rem}

\section{Systems of Hodge sheaves 
  and derived extensions}\label{sect:Section3-DBExtension}

\noin Our aim is now to use the complexes of relative logarithmic DB forms from
\autoref{sec:rel-db-cxs} to construct systems of Hodge sheaves for arbitrary flat
families. Since the proofs of \autoref{thm:DBI} and \autoref{thm:DBI'} are
interdependent, they will be presented together.  First we need to introduce the
notion of \emph{discrepancy} for flat families that appears in the setting of
\autoref{thm:DBI'}.

\begin{notation}
  Given a flat projective morphism $g: X\to B$ of regular schemes, we denote the
  reduced divisorial part of the discriminant locus $\disc(g)$ of g by $D_g$.
\end{notation}

\begin{notation}
  \label{notation:disc}
  Given a flat, projective family of schemes $f: X\to B$ with regular base $B$, let
  $\pi:\wtilde X\to X$ be a log resolution of the pair and $\wtilde f:\wtilde X\to B$
  the induced family. That is, assume that
  $\Delta_{\wtilde f}:= \wtilde f^{-1}(D_{\wtilde f})$ is a divisor with normal
  crossings. By \autoref{thm:pullback}\ref{item:6} and
  \autoref{thm:log-DB}\ref{item:4} (or \autoref{cor:pullback}), over the flat locus
  of $\wtilde f$, there is a morphism
  
  \begin{equation}\label{eq:star}
    \varkappa_p: \underline\Omega^p_{X/B}  \longrightarrow  
    \dR \pi_*  \underline\Omega^p_{\wtilde X/B}  (\log \Delta_{\wtilde f}/ D_{\wtilde
      f}) ( a_p \tightdot f^*D_{\wtilde f} ). 
  \end{equation}
  Following \autoref{def:discrep}, we use the notation $a_{\wtilde f,p}$ to denote
  the smallest integer for which the morphism \autoref{eq:star}, with the choice of
  $a_p = a_{\wtilde f,p}$, exists.
  Furthermore, we use $\af$ to denote the discrepancy $a_f(D_{\wtilde f})$ of
  $D_{\wtilde f}$ with respect to $f$ cf.~\autoref{def:discrep}.
\end{notation}

In the course of the proof of \autoref{thm:DBI} and \autoref{thm:DBI'} it is helpful
to differentiate between the various properties of systems of Hodge sheaves.  To do
so we introduce the following terminology.

\begin{definition}
  \label{def:system}
  Let $\sW$ be an $\sO_B$-module on a regular scheme $B$, $w\in\bN$.  Then a
  \emph{$\sW$-valued system of weight $w$} is a pair $(\sE, \tau)$ where $\sE$ is an
  $\sO_B$-module and $\tau: \sE \to \sW\otimes \sE$ is sheaf homomorphism, such that
  there exists an $\sO_B$-module splitting $\sE = \bigoplus_{i=0}^w \sE_i$ with
  respect to which $\tau$ is Griffiths-transversal, that is, for every $i=0,\dots,w$,
  \[
    \tau: \sE_i \to \sW \otimes \sE_{i+1} .
  \]
  Using this terminology an $\Omega^1_B$-valued system with an integrable and
  $\sO_B$-linear map $\tau$ is a \emph{system of Hodge sheaves (of weight $w$)}. When
  $\sE$ is reflexive, we call $(\sE=\bigoplus \sE_i, \theta)$ a \emph{system of
    reflexive Hodge sheaves}.
\end{definition}

\subsection{Proof of \autoref{thm:DBI} and \autoref{thm:DBI'}}
\label{subsect:System}
Let $\pi: \wtilde X \to X$ be a good resolution with the induced map
$\wtilde f: \wtilde X\to B$.  As introduced in \autoref{notation:disc}, $\af$ denotes
the discrepancy of the family with respect to $f$.  Set $m:= n-d$, where $n= \dim X$
and $d= \dim B$.  After removing a subscheme from $B$ of $\codim_B \geq 2$, defined
by the complement of the flat locus of $\wtilde f$, for every $0 \leq i \leq m$,
consider the map of distinguished triangles
\begin{equation}\label{triangle1}
  \varkappa_{m-i} : \underline \Kosz^{m-i}_{X/B}  \longrightarrow  
  \dR \pi_* \underline \Kosz^{m-i}_{\wtilde f} (\log \Delta_{\wtilde f}) \big(
  \af \tightdot f^* D_{\wtilde f} \big)  
\end{equation}
established in~\autoref{cor:pullback}. By applying $\dR f_*$ to (\ref{triangle1}) we
find
\[
  \dR f_* \varkappa_{m-i} : \dR f_* \underline \Kosz_{X/B}^{m-i} \longrightarrow \dR
  \wtilde f_* \big( \underline \Kosz_{\wtilde f}^{m-i}(\log \Delta_{\wtilde f}) \big)
  ( \af \tightdot D_{\wtilde f} ) .
\]
From the resulting cohomology sequence and the filtered quasi-isomorphism 
\[
  ( \underline \Omega^{m-i}_{\wtilde X}(\log \Delta_{\wtilde f}) , \bF^{\kdot}_f )
  \simeq_{qis} ( \Omega^{m-i}_{\wtilde X} (\log \Delta_{\wtilde f}) , F^{\kdot}_K )
\]
in \autoref{thm:log-DB}\autoref{item:5}, and \autoref{rem:triangles} we find
connecting homomorphisms $\tau_i$ and $\theta_i^0 \otimes \id$ with the commutative
diagram
\begin{equation}\label{eq:ComSquare}
  \begin{aligned}
    \xymatrix{ \myR^i f_* \bGr_{\bF_{X/B}^{\kdot}}^0 \ar[d] \ar[rr]^{\tau_i} &&
      \myR^{i+1}
      f_* \bGr^1_{\bF_{X/B}^{\kdot}}  \ar[d] \\
      \myR^i \wtilde f_* \bG r^0_{\wtilde f} \; (\af\tightdot
      D_{\wtilde f}) \ar[d]^{\simeq} \ar[rr] && \myR^{i+1} \wtilde f_* \bG
      r^1_{\wtilde f} \; ( \af \tightdot
      D_{\wtilde f} )\ar[d]^{\simeq}  \\
      \myR^i \wtilde f_* \Gr^0_{F_K^{\kdot}} (\af \tightdot D_{\wtilde
        f}) \ar[rr]^{\theta_i^0 \otimes \id} && \myR^{i+1} \wtilde f_*
      \Gr^1_{F_K^{\kdot}} (\af \tightdot
      D_{\wtilde f})  .   \\
    }
  \end{aligned}
\end{equation}
Here, by
$\theta_i^0: \dR^i \wtilde f_* \Gr^0_{F^{\kdot}_K} \to \dR^{i+1}\wtilde f_*
\Gr^1_{F^{\kdot}_K}$ we denote the connecting map arising from the cohomology
sequence associated to $\dR \wtilde f_* \Kosz^{m-i}_{\wtilde f} (\log \Delta)$.

Next, we define the two systems $(\sF= \bigoplus \sF_i, \tau = \bigoplus \tau_i)$ and
$(\sE^{\bf 0} = \bigoplus \sE_i^{\bf 0}, \theta^{\bf 0} = \bigoplus \theta^{\bf 0}_i
)$ by
\begin{equation}\label{eq:SysDefs}
  \sF_i: = \myR^i f_* \underline \Omega^{m-i}_{X/B}  \;\;\; \text{and} \;\;\;
  \sE_i^{\bf 0}: = \myR^i \wtilde f_* \Omega^{m-i}_{\wtilde X/B} (\log \Delta_{\wtilde
    f}).  
\end{equation}
By construction, the two systems $(\sF, \tau)$ and $(\sE^{\bf 0}, \theta^{\bf 0})$
are $\Omega^1_B$-valued and $\Omega^1_B (\log D_{\wtilde f})$-valued, respectively.
Set
\[
  (\sE^{\af 
  } , \theta^{\af} ) : = (\sE^{\bf 0} ,
  \theta^{\bf 0}) (\af \tightdot D_{\wtilde f} ) .
\]
It follows from (\ref{eq:ComSquare}) that there are sheaf morphisms
$\psi_i: \sF_i \to \sE_i^{\af}$ fitting in the commutative diagram,
\[
  \xymatrix{ \sF_i \ar[rr]^{\tau_i} \ar[d]^{\psi_i} && \Omega^1_B \otimes \sF_{i+1}
    \ar[d]^{\psi_{i+1}}  \\
    \sE_i^{\af} \ar[rr]^(0.4){\theta^{\af}_i} &&
    \Omega^1_B (\log D_{\wtilde f}) \otimes \sE_{i+1}^{\af} .  }
\]
Define the system map 
\begin{equation}\label{eq:psi}
  \psi = \bigoplus  \psi_i : (\sF, \tau) \longrightarrow (\sE^{\af},
  \theta^{\af}), 
  \end{equation}
  and let $(\sG=\bigoplus \sG_i, \theta_{\sG})$ denote its image.

  By extending the result of Katz-Oda \cite{KO68} to the case of logarithmic relative
  de Rham complex and using \cite{Ste76}*{2.18} (cf.~\cite{Gri84}*{p.131}) one finds
  that $(\sE^0, \theta^0)$ is the logarithmic system of Hodge bundles underlying the
  Deligne canonical extension of the flat bundle
  $( \myR^m f_* \bC_{\wtilde X\setminus \Delta_{\wtilde f}} \otimes \sO_{B\setminus
    D_{\wtilde f}} , \nabla )$, with $\nabla$ denoting the Gauss-Manin connection.
  In particular $\theta^{\af}$ is $\sO_B$-linear and integrable.  Consequently, so is
  $\theta_{\sG}$, that is $(\sG, \theta_{\sG})$ is a system of Hodge sheaves.
  Furthermore, with $(\sE^{\af}, \theta^{\af})$ being locally free, the morphism
  $\psi: \sF \to \sE^{\af}$ factors through $\sG \to \sG^{**}$. We now define
\begin{equation}\label{eq:extension}
  ( \overline \sE, \overline \theta ) : = ( \sG , \theta_{\sG})^{**} . 
\end{equation}

The last part of \autoref{thm:DBI} follows from the construction of
$(\overline \sE, \overline \theta)$ and \autoref{thm:log-DB}\autoref{item:8}.  More
precisely, we have that
$\sF_0 \simeq \dR^0 f_* \underline\Omega^m_{X/B} \simeq f_*\omega_{X/B}$.
Furthermore, using the isomorphism $\pi_* \omega_{\wtilde X} \simeq \omega_X$, we
find that
\[
  \psi_0 : \sF_0 |_{B\setminus D_{\wtilde f}} \longrightarrow \sE_0^{a_{\wtilde
      f}}|_{B\setminus D_{\wtilde f}}
\]
is an isomorphism.  As $\sF_0$ is torsion free, this implies that $\psi_0$ is
injective. Therefore, $\sF_0$ can be identified with its image under $\psi_0$, which
we have denoted by $\sG_0$. In particular we have
$\sG_0^{**} \simeq \sF_0^{**} \simeq (f_*\omega_{X/B})^{**}$.  This completes the
proof of \autoref{thm:DBI} and \autoref{thm:DBI'}.  \qed

\subsection{Explanation for \autoref{thm:DBII}}
Let $f_U: U\to V$ be the smooth locus of
$f: X\to B$ and $i: U \to X$ and $j: V\to B$ the natural inclusion maps.

\begin{claim}\label{claim:ISOM}
$\overline \sE_l|_V  \simeq \myR^l f_* \Omega^{m-l}_{U/V}$ .
\end{claim}

\noindent
\emph{Proof of Claim~\ref{claim:ISOM}.}
This directly follows from flat base change and properties of complexes of relative
DB forms.  More precisely, we have

\begin{align*}
  \pushQED{\qed} 
  \overline \sE_l |_V    &  \simeq  j^* \myR^l f_*  \underline \Omega^{m-l}_{X/B}  
                           \text{, by the definitions of $\overline \sE_l$ in
                           (\ref{eq:SysDefs}) and $\psi$ in (\ref{eq:psi})}\\ 
                         &  \simeq  \myR^l f_* ( i^*  \underline \Omega^{m-l}_{X/B} )
                           \text{, using flat base change}\\ 
                         &   \simeq  \myR^l f_* \underline \Omega_{U/V}^{m-l}   
                           \text{, according to the construction}  \\ 
                         &   \simeq \myR^l f_*  \Omega^{m-l}_{U/V}   
                           \text{, by the quasi-isomorphism in
                           \autoref{thm:log-DB}\autoref{item:5}}
                           .  
                           \qedhere
                           \popQED
\end{align*}

Now, by construction we have
$i^* \underline \Kosz_{X/B}^p \simeq \underline\Kosz^p_{U/V}$.  Moreover, by
\autoref{rem:triangles}, we have a natural isomorphism of triangles
$\underline \Kosz^p_{U/V} \to \Kosz^p_{U/V}$, inducing the isomorphism
\begin{equation}\label{eq:LAST}
j^*\myR f_* \underline\Kosz^p_{X/B}  \longrightarrow  \myR f_* \Kosz^p_{U/V}  .
\end{equation}
Thanks to \cite{KO68} we know that the Higgs field underlying $\nabla$ (the
Gauss-Manin connection) is defined by the connecting homomorphisms of the long exact
cohomology sequence associated to the right-hand-side of \autoref{eq:LAST}. On the
other hand, by the construction of $(\overline \sE, \overline \theta)$, base change,
and \autoref{claim:ISOM}, the left-hand-side of \autoref{eq:LAST} similarly
determines $\overline \theta|_V$ in $(\overline \sE,\overline \theta)|_V$.
Therefore, we find that $(\overline\sE, \overline \theta)$ defines an extension of
the Hodge bundle underlying $\nabla$.  \qed

\subsection{Functorial properties} 
\label{ssect:funct}
The proof of \autoref{thm:DBII} already exhibits some of the 
functorial properties of 
the construction of $(\overline \sE, \overline \theta)$ in \autoref{thm:DBI}. 
This can be further formalized in the following way. 

Let $\mathfrak{Fam}(n,d)$ be the category of 
projective surjective morphisms $f:X\to B$ with connected fibers between an
$n$-dimensional reduced scheme $X$ and a smooth quasi-projective scheme $B$ of
dimension $d$. A morphism $(f': X'\to B') \to (f:X\to B)$ in $\mathfrak{Fam}(n,d)$ is
defined by a commutative diagram
$$
\xymatrix{
X'  \ar[rr]^{\gamma'}  \ar[d]_{f'}  &&   X  \ar[d]^f  \\
B'   \ar[rr]^{\gamma}           &&   B .
}
$$

Further let $\mathfrak{Hodge}(d,w)$ denote the category of triples
$(B, \sE, \theta)$, where $(\sE, \theta)$ is a system of reflexive Hodge sheaves of
weight $w$ on the smooth quasi-projective scheme $B$ of dimension $d$. A morphism
$\Gamma: (B', \sE,\theta') \to (B, \sE, \theta)$ in this category consists of a
morphism $\gamma: B'\to B$, such that the induced morphism
$\sE\to \dR \gamma_* \sE'$, fits into the commutative diagram
$$
\xymatrix{
\sE \ar[rr]  \ar[d]_{\theta}  &&   \dR \gamma_* \sE'  \ar[d]^{\dR \gamma_* \theta}  \\
\Omega^1_B\otimes \sE   \ar[rr]           &&   \dR\gamma_* (\Omega^1_B\otimes \sE').
}
$$

\subsubsection{Proof of \autoref{thm:funct}}
This directly follows from the construction of Koszul triangles in \ref{subsect:Koszul}
and the system $(\overline \sE, \overline \theta)$ in \ref{subsect:System}. 
More precisely, consider a log-resolution $\pi: \wtilde X\to X$
as in \autoref{notation:disc} . 
Similarly, let $\pi': \wtilde{X'} \to X'$ be a log-resolution 
factoring through the projection $\wtilde X \times_X X' \to X'$. By $\wtilde{f'}: \wtilde{X'}\to B'$ 
we denote the induced 
family. 
By construction, we have a commutative diagram of distinguished triangles 

$$
\xymatrix{ \underline \Kosz^p_{X/B} \ar[d]_{\varkappa_p} \ar[rr] && \dR
  \gamma'_*(\underline \Kosz^p_{X'/B'}) (a_f \cdot f^*D_{\wtilde f})
  \ar[d]^{\dR \gamma'_*\varkappa_p} \\
  \dR \pi_*\underline \Kosz^p_{\wtilde X/B} (\log \Delta_{\wtilde {f'}} ) \ar[rr] &&
  \dR \gamma'_* \big( \dR \pi'_* \underline \Kosz^p_{\wtilde{X'}/B'} (\log
  \Delta_{\wtilde {f'}} ) (a_{f'} \cdot {f'}^*D_{\wtilde{f'}})\big) .  }
$$
The theorem now follows by applying $\dR f_*$ to this diagram.

\section{Positivity of direct image sheaves and the discrepancy of the family}
\label{sect:Section4-Application}
\noindent
We will continue using the notation introduced in \autoref{notation:disc}.  First we
show a somewhat weaker version of \autoref{thm:Cons}.

\begin{proposition}
  \label{thm:App}
  Let $X$ and $B$ be two projective varieties of dimension $n$ and $d$, respectively,
  $f:X\to B$ a flat family of geometrically integral varieties with only Gorenstein
  Du~Bois singularities, such that $B$ is smooth and the generic fiber of $f$ has
  rational singularities.  Further let $D\subset B$ be an effective divisor
  satisfying $D\leq D_{\wtilde f}$, $D':= D_{\wtilde f}- D$, and let
  $r:= \rank(f_*\omega_{X/B})$.  Then, one of the following holds.
  \begin{enumerate}
  \item Either
    $c_1(\det (f_*\omega_{X/B}) ( - rn  D - D')) \cdot H^{d-1} \leq 0$, for
    some ample divisor $H\subset B$, or
  \item there exists a pseudo-effective line bundle $\sB$ on $B$ for which there is
    an injection
    \begin{equation*}
      \big(  (\det f_* \omega_{X/B} )  (-rn  D - D') \big)^t  \otimes \sB
      \hooklongrightarrow  (\Omega^1_B (\log  D') )^{\otimes N} , 
    \end{equation*}
    for some $t,N \in \bN$.
  \end{enumerate}
\end{proposition}

\begin{remark}
  The above result remains valid if we replace $nr$ by the discrepancy of
  $D_{\wtilde f}$ with respect to $f^r$. We opted to avoid a cumbersome notation, and
  instead use the upperbound $rn$, cf.~\autoref{thm:DBI'}.
\end{remark}

\begin{remark} 
  One may also replace $f_*\omega_{X/B}$ in \autoref{thm:App} by any of its
  subsheaves (and of course replace $r$ with the corresponding rank).  This is of
  interest for example in the setting of Fujita's Second Main Theorem (see
  \cite{CD17} and references therein).
\end{remark}

Before proving~\autoref{thm:App}, we recall the following well-known fact
regarding the functoriality of canonical extensions.

\begin{fact}\label{Functorial}
  Let $f:X\to B$ be a projective morphism of smooth quasi-projective varieties $X$
  and $B$.  Assume that $D_f$ and $\Delta:= f^{-1} D_f$ are simple normal crossing
  divisors.  Let $\gamma: C\to B$ be a morphism of smooth quasi-projective
  varieties. Let $X_C$ be a strong resolution of $X\times_B C$, with $f_C: X_C \to C$
  being the naturally induced family.  Assume that the support of $D_{f_C}$ and
  $\Delta_{f_C}$ are simple normal crossing divisors.  Let
  $(\sE^0_C = \bigoplus \sE^0_{C,i}, \theta^0_C)$ be the associated system of Hodge sheaves underlying 
  Deligne extension of the
  local system $\R^j f_* \bC_{X_C \setminus \Delta_{f_C}}$ of any given weight
  $j$. Then, as systems of Hodge sheaves, we have an inclusion

  \[
    \gamma ^* (\sE^0, \theta^0) \subseteq ( \sE^0_C, \theta^0_C ) ,
  \]
  which is an isomorphism over the flat locus of $\gamma|_{C\setminus D_{f_C}}$.
\end{fact}

\noindent
\emph{Proof of \autoref{thm:App}.}  Let
$\sL: = (\det f_*\omega_{X/B}) (- r n D - D')$ and assume that for some ample
divisor $H\subset B$ we have
\begin{equation}\label{Assump}
  c_1(\sL) \cdot H^{d-1}>0 .
\end{equation}
Let $X^r$ denote the $r$th fiber product $X\times_B \dots \times_B X$ ($r$ times)
with the resulting morphism $f^r: X^r\to B$.  Now, as $f$ is Gorenstein and flat we
have
\[
  \bigotimes^r f_*\omega_{X/B} \simeq f_*^r \omega_{X^r/B}.
\]
Note, that $X^r$ has rational singularities by \cite{KS13}*{Theorem~E}
(cf.~\cite{Elkik78}).

Let $n':= n-d$. By slightly modifying $\sF_i$ in \autoref{eq:SysDefs} we set 
$\sF_i:= \dR^i f^r_* ( \underline \Omega^{rn'-i}_{X^r/B} (\log \Delta'/D') )$, 
where $\Delta':= f^{-1} D'$.
Since $X^r$ has only rational singularities, by 
\autoref{thm:log-DB}\autoref{item:8} we have 
$$  f^r_* \omega_{X^r/B}\big( \Delta' - (f^r)^* D' \big)    \hookrightarrow \sF_0  .$$
From the natural inclusion
$\det f_*\omega_{X/B} \subseteq f^r_* \omega_{X^r/B}$ it follows that there is an
injection
\begin{equation}\label{INJ}
  \sL \longrightarrow \sF_0 (- rn D) .
\end{equation}

To simplify our notation, we will replace $f$ by $f^r$ in the sequel. Similar to
\autoref{triangle1} we have a morphism of triangles (\autoref{rk:refined})
\[
  \varkappa_{rn'-i}: \underline\Kosz_f^{rn'-i} (\log \Delta') \longrightarrow \dR
  \pi_* \underline \Kosz^{rn'-i}_{\wtilde f} (\log \Delta_{\wtilde f}) \otimes
  f^*(rn D),
\]
inducing a morphism of systems 
\[  
  \psi = \bigoplus \psi_i : (\sF, \tau) ( - rn D ) \longrightarrow (\sE^0, \theta^0),
\]
whose weight is equal to the relative dimension. 
Denote the image of $\psi$ by
$(\sG=\bigoplus \sG_i, \theta) \subseteq (\sE^0, \theta^0)$,
$\theta: \sG_i \to \Omega^1_B(\log D')\otimes \sG_{i+1}$.  By construction
$\psi_0$ is injective and thus $\sL \hooklongrightarrow \sG_0$.

\begin{claim}\label{claim:maps}
  $\theta(\sL) \neq 0$.
\end{claim}

\begin{proof}[{Proof of \autoref{claim:maps}}]
  Aiming for a contradiction, assume that $\theta ( \sL) =0$.  Let
  $\overline \sL$ denote the saturation of the image of the injection
  \autoref{INJ}. Set $C\subseteq B$ to be the smooth, complete intersection curve
  defined by $H^{d-1}$ with the natural inclusion map $\gamma: C\to B$.  For a
  suitable choice of $C$ we can ensure that $C$ is in the locus of $B$ over which
  $\overline\sE_0 / \overline \sL$ is locally free.  Let $(\sE_C, \theta_C)$ be the
  logarithmic Hodge system defined in \autoref{Functorial}. According to
  \autoref{Functorial} we have an injection
  \begin{equation}\label{eq:compatible}
    \gamma^* (\sE^0 = \bigoplus \sE_i^0 , \theta^0)  \hooklongrightarrow  (\sE_C^0 =
    \bigoplus \sE^0_{C,i}, \theta_C^0). 
  \end{equation}
  In particular we have an injection 
  \[
    \gamma^* \overline \sL \hooklongrightarrow \sE^0_{C,0} .
  \]
  From our initial assumption it follows that $\theta_{C,0} (\gamma^*\overline \sL) =0$.
  On the other hand, since $\ker ( \theta^0_C |_{\sE^0_{C,i}} )$ is weakly negative
  by \cite{Zuo00}, this implies that $\deg(\gamma^* \overline \sL|_C ) \leq 0$ and thus
  contradicting our initial assumption \autoref{Assump}. This finishes the proof of
  the claim.
\end{proof}

Now, by applying $\theta$ to $\sL$ we can find an integer $k\geq 1$ such that
\[
  \underbrace{(\id \otimes \theta) \circ \dots \circ (\id \otimes
    \theta)}_{\text{$k-1$ times}} \otimes \theta :\sL \hooklongrightarrow
  (\Omega^1_B(\log D'))^{\otimes k} \otimes \sN_k ,
\]
where $\sN_k: = \ker ( \theta|_{\sG_k} )$.  
As $\sN_k\subseteq \ker (\theta^0|_{\sE^0_k})$
and since $\ker (\theta^0|_{\sE^0_k})$ is weakly
negative \cite{Zuo00}, we find that there is an injection
$\sL^t \otimes \sB \hooklongrightarrow (\Omega^1_B (\log D'))^{\otimes N}$, for some $t\in \bN$
and pseudo-effective $\sB:= (\det \sN_k)^{-1}$.

\subsection{The general case}
\label{subsec:case-mth-direct}
By using a cyclic covering construction (see also \cite{Vie-Zuo03a}, \cite{Taj18} and
\cite{Taj20}), combined with the constructions in
\autoref{sect:Section3-DBExtension}, we generalize \autoref{thm:App} to the
pluricanonical case.

\begin{theorem}\label{thm:TheEnd}
  Using the notation and assumptions of \autoref{thm:App}, for every $m\in \bN$, let
  $r_m =\rank( f_*\omega^m_{X/B} )$ and set $t_m: = m r_m n$.  Then, for any
  $a\in \bN$, either
  \begin{enumerate}
  \item
    $c_1( (\det f_*\omega^m_{X/B})^a( - a t_m  D - D' )) \cdot H^{d-1} \leq 0 $,
    or
  \item \label{2Inj}there exists a pseudo-effective line bundle $\sB$ on $B$ for
    which there exists an injection
    \[
      \big( ( \det f_*\omega^m_{X/B} )^a( - a t_m D - D' ) \big)^t \otimes \sB
      \hooklongrightarrow (\Omega^1_B (\log D') )^{\otimes N},
    \]
  \end{enumerate}
  for some $t, N \in \bN$.
\end{theorem}

\begin{proof}
  Using the notation of the proof of \autoref{thm:App}, consider the natural
  injection
  \[
    \big( \det f_*\omega^m_{X/B} \big)^{ma} \hooklongrightarrow f_*^{m r_m a}
    \omega^m_{X^{m r_m a}/B},
  \]
  which, with $\sA_m: =(\det f_*\omega^m_{X/B})^a$, implies that for the line bundle
  $\sM$ defined by
  \[
    \sM : = \omega_{X^{m r_m a}/B} \otimes ( f^{m r_m a} )^* ( \sA_m )^{-1}
  \]
  we have $H^0( X^{m r_m a} , \sM^m )\neq 0$.
  Following \autoref{notation:disc}, $\pi:\wtilde X\to X$ denotes a good resolution
  of $X$ and further let $\wtilde X^{(m r_m a)}$ denote a strong resolution of
  $(\wtilde X)^{m r_m a}$.  The composition of this latter resolution with
  $\pi^{m r_m a}$ induces a projective birational morphism
  $\mu: \wtilde X^{(m r_m a)} \to X^{m r_m a}$.

  Note that the assumptions on the singularities of the fibers of $f$ remain true for
  the fibers of $f^{m r_m a}$, so
  for ease of notation, let us replace $f^{m r_m a}$ by $f$ and, similarly, replace
  $\wtilde X^{(m r_m a)}$ by $\wtilde X$.

  Next, define $n':= n-d$, $t'_m:= m r_m n'$ and modify the system $(\sF, \tau)$,
  defined in \autoref{eq:SysDefs}, by setting
  \[
    \sF_i' : = \R^i f_* ( \underline \Omega_{X/B}^{at'_m -i} (\log \Delta'/D')
    \otimes \sM^{-1} ) ,
  \]
  where $\Delta': = f^{-1} (D')$. We get a morphism of triangles
  \begin{equation}\label{eq:STAR}
    \begin{aligned}
      \varkappa_{at'_m-i} : \underline \Kosz_f^{at'_m-i} (\log \Delta') \otimes
      \sM^{-1} & \longrightarrow \\ \longrightarrow \dR \mu_* & \big( \underline
      \Kosz^{at'_m-i}_{\wtilde f} (\log \Delta_{\wtilde f}) \otimes \mu^*\sM^{-1}
      \big) (at_m\, f^* D),
    \end{aligned}
  \end{equation}
  cf.~ \autoref{rk:refined}.  Now, let $\sigma: Z\to \wtilde X$ be a resolution of
  singularities of the cyclic covering associated to a global section of
  $\mu^* \sM^m$ (cf.~\cite{Laz04-I}*{4.1.6}) and let $g: Z\to B$ be the induced
  map. 
  After removing a subscheme of $B$ of $\codim_B\geq 2$ we may assume 
  that $g:(Z, \Delta_g) \to (B,D_g)$ is snc. 
  
  \begin{claim}\label{anti-effective}
  In the setting above, there is a natural morphism of triangles 
  \begin{equation}\label{eq:2ndTri}
    \underline \Kosz^{at'_m-i}_{\wtilde f} (\log \Delta_{\wtilde f}) \otimes \mu^* \sM^{-1} 
    \longrightarrow  \dR \sigma_*   \underline \Kosz_g^{at'_m-i} (\log \Delta_g)
    .
  \end{equation}
  \end{claim}
  
  \noindent
  \emph{Proof of \autoref{anti-effective}.}
  By using \autoref{rk:refined} with $\Delta_g$,
  $\Delta_{\wtilde f}$, $D_g$ and $D_{\wtilde f}$ playing the role of $\Gamma$,
  $\Delta=\Delta'$, $D$ and $D'$, and the fact that $\af(D_g-D_{\wtilde f}) =0$, we
  get a morphism of triangles
  
  \begin{equation}\label{eq:yek}
  \underline \Kosz^{at'_m-i}_{\wtilde f}  (\log \Delta_{\wtilde f})  \longrightarrow  \dR \sigma_*  \underline\Kosz^{at'_m-i}_g (\log \Delta_g).
  \end{equation}
  Since the associated morphisms are snc, the two Koszul triangles in \autoref{eq:yek} are isomorphic 
  to the two complexes (short exaxt sequences) of locally free sheaves $\Kosz^{at'_m-i}_{\wtilde f}  (\log \Delta_{\wtilde f})$,  
  $\Kosz^{at'_m-i}_g (\log \Delta_g)$, 
  with the morphism \autoref{eq:yek} naturally arising from 
  \begin{equation}\label{eq:DO}
  \sigma^* \Kosz_{\wtilde f}^{at'_m-i} (\log \Delta_{\wtilde f})  \longrightarrow  \Kosz^{at'_m-i}_g (\log \Delta_g). 
  \end{equation}
  On the other hand, by the construction of $\sigma$ we have $h^0( \sigma^* \mu^* \sM  ) \neq 0$, i.e. 
  $\sigma^*\mu^* \sM^{-1} \hookrightarrow \sO_Z$. 
  Combining this with \autoref{eq:DO} we find 
  $$
  \sigma^* \big(  \Kosz_{\wtilde f}^{at'_m-i} (\log \Delta_{\wtilde f})  \otimes \mu^*\sM^{-1}  \big)
       \longrightarrow   \Kosz^{at'_m-i}_g (\log \Delta_g).
  $$
  Applying $\dR \sigma_*$, the projection formula gives  
  $$
   \Kosz_{\wtilde f}^{at'_m-i} (\log \Delta_{\wtilde f})  \otimes \mu^*\sM^{-1}   \longrightarrow  
   \dR \sigma_* \Kosz^{at'_m-i}_g (\log \Delta_g)  ,
  $$
  as required. \qed
 
  Combining \autoref{eq:2ndTri} and \autoref{eq:STAR} leads to the morphism of
  triangles:
  \[
    \underline \Kosz_f^{at'_m-i} (\log \Delta') \otimes \sM^{-1} \longrightarrow \dR
    \eta_* \big( \underline \Kosz_g^{at'_m-i} (\log \Delta_g) \big) \otimes (at_m
    \, f^*D),
  \]
  where $\eta: = \mu \circ \sigma$.  Similarly to \autoref{eq:psi} it follows that
  there exists a morphism of systems of equal weight (equal to the relative dimension
  of $f$):
  \[
    \psi^{\sM} = \bigoplus \psi_i^{\sM} : (\sF', \tau') (-at_m D )
    \longrightarrow (\sE^0, \theta^0) ,
  \]
  where $(\sE^0, \theta^0)$ is the system underlying the Deligne extension for
  $\dR^{at'_m} g_* \bC_{Z\setminus \Delta_g}$.  On the other hand, since $X$ has
  rational singularities, by \autoref{thm:log-DB}\autoref{item:8} and construction we have 
  $\sA_m(-D') \hookrightarrow \sF'_0$. 
  Consequently we find 
  $$
  \sA_m(-D') (-at_m D) \hookrightarrow \sF'_0 (-at_m D)
  \xhookrightarrow{\psi^{\sM}_0} \sG_0 \subseteq \sE_0 ,
  $$
  where $(\sG=\bigoplus \sG_i, \theta)$ is the image of $(\sF', \tau')(-at_m D)$
  under $\psi^{\sM}$.  The rest of the argument follows as in the $m=1$ case
  (\autoref{thm:App}).
\end{proof}

\noindent
The following corollary now directly follows from \autoref{thm:TheEnd} and
\cite{CP16}*{7.11}.
 
\begin{corollary}
  [= \autoref{thm:Cons}]
  \label{cor:END}
  Using the notation and assumptions of \autoref{thm:TheEnd}, further assume that
  $(\det f_*\omega^m_{X/B})(- t_m D)$ is big. Then, the pair $(B, D_{\wtilde f}- D)$
  is of log-general type.
\end{corollary}

\begin{proof}
Since $(\det f_*\omega^m_{X/B})(- t_m  D)$ is big, for any sufficiently large $a\in \bN$, 
the line bundle $(\det f_*\omega^m_{X/B})^a(-at_m D)(-D')$ is also big. 
Therefore, for a pseudo-effective line bundle $\beta$ and $t,N \in \bN$, 
there is an injection as in \autoref{thm:TheEnd} \autoref{2Inj}.  
The rest now follows from \cite{CP16}*{7.11}.
\end{proof}

\section{Appendix: Summary of wedge products and filtration
  diagrams}\label{sec:appendix}

\newcommand{\ww}[2]{\curlywedge _{#1}^{#2}}%
\newcommand{\wws}[2]{\scriptstyle \curlywedge _{#1}^{#2}}%

\newcommand\scat{\frS\frc\frh_{S}}%
\newcommand\slf{\frL\fro\frc\frF\frr\fre\fre_{S}}

\newcommand\nsmorphd{Let $\Phi_X$ and $\Psi_X$ be locally free sheaves on $X$ of rank
  $k$ and $n$ respectively, and let $\theta_X:\Phi_X\to\Psi_X$ be a morphism}

\newcommand\nsmorphdfunc{Let $\Phi$ and $\Psi$ be honest functors from $\scat$ to
  $\slf$, and $\theta:\Phi\to\Psi$ a natural transformation}

\newcommand\mywedge{\displaystyle\bigwedge}

\newcommand\Ox{\Psi_X}
\newcommand\Oxi[1]{\mywedge^{#1}\Psi_X}
\newcommand\Oxs{{\frQ}}
\newcommand\Oxsi[1]{{\frQ}^{#1}}
\newcommand\Oxsui[1]{\underline{\frQ}^{#1}_{\,\theta_X}} 
\newcommand\Oxsu{\underline{\frQ}}
\newcommand\fos{\Phi_X}
\newcommand\tfos{\otimes \Phi_X}
\newcommand\foms{\det\Phi_X}
\newcommand\tfoms{\otimes \det\Phi_X}
\newcommand\fosi[1]{\mywedge^{#1}\Phi_X}
\newcommand\tfosi[1]{\otimes \mywedge^{#1}\Phi_X}
\newcommand\fomsi[1]{(\det\Phi_X)^{#1}}
\newcommand\fomsj[1]{(\det\Phi_X)^{#1}}
\newcommand\tfomsi[1]{\otimes (\det\Phi_X)^{\otimes #1}}
\newcommand\tfomsj[1]{\otimes (\det\Phi_X)^{#1}}
\newcommand\Osi[1]{\mywedge^{#1}\Phi_X}
\newcommand\Oy{\Psi_Y}
\newcommand\Oyu{\underline\Psi_Y}
\newcommand\Oysu{\underline{\frQ}_{\,\nattr_Y}} 
\newcommand\Oysui[1]{\underline{\frQ}^{#1}_{\,\nattr_Y}} 
\newcommand\Oyi[1]{\mywedge^{#1}\Psi_Y}
\newcommand\Oyui[1]{\mywedge^{#1}\Psi_Y}
\newcommand\ff[2]{\frF_{#1}^{#2}}
\newcommand\ffb[2]{\boxed{\frF_{#1}^{#2}}}
\newcommand\ffbi[3]{\boxed{\frF_{#1}^{#2}}^{\ #3}}
\newcommand\ffbx[3]{\boxed{\frF_{#1}^{#2}(#3)}}
\newcommand\ffs[2]{\scriptstyle\frF_{#1}^{#2}}
\newcommand\ffsb[2]{\boxed{\ffs{#1}{#2}}}
\newcommand\ffsbx[3]{\boxed{\scriptstyle\frF_{#1}^{#2}(#3)}}
\newcommand\fa[2]{\frf_{#1}^{#2}}
\newcommand\fab[2]{\boxed{\frf_{#1}^{#2}}}
\newcommand\ffa[2]{\frF\fri\frl\frt_{#1}^{#2}}
\newcommand\ffab[2]{\boxed{\frF\fri\frl\frt_{#1}^{#2}}}
\newcommand\hypf[1]{\mathbb F^{#1}\hskip-2pt}
\newcommand\hypg[1]{\mathbb G^{#1}\hskip-2pt}
\newcommand\hype[1]{\mathbb E^{#1}\hskip-2pt}
\newcommand\hypgre[1]{\mathbb Gr^{#1}_{\mathbb E}\hskip-2pt}
\newcommand\hypgr[1]{\mathbb Gr^{#1}_{\mathbb F}\hskip-2pt}
\newcommand\hypgrr[2]{\mathbb Gr^{#1}_{#2}\hskip-2pt}

\newcommand\dist[3]{\xymatrix{#1\ar[r] & #2 \ar[r] & #3\ar[r]^-{+1} & }}
\newcommand\shortses[3]{0 \rightarrow #1 \rightarrow #2 \rightarrow #3 \rightarrow 0}
\newcommand\shortdist[3]{\dist{#1}{#2}{#3}}

\noin
We recall some definitions and constructions from \cite{Kovacs05a} for the reader's
convenience.

\subsection{Wedge products}
\nsmorphd.

\begin{definition}\label{wwdef}
  Let $\eta$ be a section of $\Oxi p$ over an open set and $\xi_1,\dots,\xi_k$
  sections of $\fos$ over the same set. Then $\eta\otimes(\xi_1\land\dots\land\xi_k)$
  is a section of $\Oxi p\otimes\foms$. For any $\sigma\in S_k$ let
  \[
    \xi_{\sigma,
      q}=\theta_X(\xi_{\sigma(1)})\land\dots\land\theta_X(\xi_{\sigma(q)}),
  \]
  and
  \[
  \xi^{\sigma, q}=\xi_{\sigma(q+1)}\land\dots\land\xi_{\sigma(k)}.
  \]

  Further let
  \[
    S_{k, q}= \{\sigma\in S_k | \sigma(1)<\dots<\sigma(q)\text{ and
    }\sigma(q+1)<\dots<\sigma(k)\},
  \] and
  \[
  I_{\sigma}=\{\sigma(1),\dots,\sigma(q)\}.
  \] 

  It is easy to see that every $\sigma\in S_{k, q}$ is determined by $I_\sigma$. Now
  define
  \[
    \ww q\theta(\eta\otimes(\xi_1\land\dots\land\xi_k))\in\Oxi{p+q}\otimes\fosi {k-q}
  \]
  by the formula
  \[
    \ww q\theta(\eta\otimes(\xi_1\land\dots\land\xi_k))= \sum_{\sigma\in S_{k, q}}
    (-1)^{\sgn\sigma}(\xi_{\sigma, q}\land\eta)\otimes\xi^{\sigma, q},
  \]
  and extend it linearly.
\end{definition}

To see that
\[
  \ww q\theta: {\Oxi p\otimes\foms}\to{\Oxi {p+q}\otimes\fosi {k-q}}
\]
is a well-defined morphism of sheaves, it is enough to verify the multi-linear and
alternating properties. This is left to the reader.

\begin{lemma}\label{wedgecomm}
  Let $\id$ denote $\id_{\Phi_X}:\Phi_X\to\Phi_X$. Then
  \[
    \xymatrix{%
      \ar[d]_{\ww {q+r}\theta} \Oxi p\otimes\foms\otimes\foms \ar[r]^{\ww q\theta} &
      \Oxi {p+q}\otimes \foms\otimes \fosi {k-q}
      \ar[d]^{\ww r\theta}\\
      \Oxi {p+q+r}\otimes \fosi {k-q-r}\otimes \foms \ar[r]^{\ww q\id} &
      \Oxi {p+q+r}\otimes \fosi {k-r}\otimes \fosi {k-q} \\
    }
  \]
  is a commutative diagram, i.e.,
  $ \ww r\theta\circ\ww q\theta=\ww q\id\circ\ww {q+r}\theta$.
\end{lemma}

\begin{proof}
  Use the same notation as in \autoref{wwdef}. Then
  \begin{multline*}
    \ww r\theta\circ\ww q\theta(\eta\otimes(\xi_1\land\dots\land\xi_k)
    \otimes(\xi_1\land\dots\land\xi_k))=\\
    =\sum_{\tau\in S_{k, r}}\sum_{\sigma\in S_{k, q}}
    (-1)^{\sgn\tau+\sgn\sigma}(\xi_{\tau,r}\land\xi_{\sigma, q} \land \eta)\otimes
    \xi^{\sigma,q}\otimes\xi^{\tau,r}.
  \end{multline*}

  Let $\sigma\in S_{k,q}$, $\tau\in S_{k, r}$. If
  \[
    I_\tau\cap I_\sigma=\{\tau(1),\dots,\tau(r)\}\cap
    \{\sigma(1),\dots,\sigma(q)\}\neq\emptyset,
  \]
  then $\xi_{\tau,r}\land \xi_{\sigma, q}=0$. Otherwise let
  $\mu=\mu(\sigma, \tau)\in S_{k, q+r}$ be defined by $I_\mu=I_\tau\cup I_\sigma$ and
  let $\nu=\nu(\sigma, \tau)=\sigma\in S_{k,q}$.
  \begin{multline*}
    \ww q\id\circ\ww {q+r}\theta(\eta\otimes(\xi_1\land\dots\land\xi_k)
    \otimes(\xi_1\land\dots\land\xi_k))=\\
    \sum_{\nu\in S_{k, q}}\sum_{\mu\in S_{k, q+r}}(-1)^{\sgn\mu+\sgn\nu}(\xi_{\mu,
      q+r}\land\eta)\otimes\xi^{\nu,q}\otimes (\xi_{\nu,q}\land\xi^{\mu,q+r})
  \end{multline*}
  and for $\nu\in S_{k, q},\ \mu\in S_{k, q+r},\ \xi_{\nu,q}\land\xi^{\mu,q+r}\neq 0$
  let $\sigma=\sigma(\mu, \nu)=\nu$ and $\tau=\tau(\mu, \nu)\in S_{k,r}$ be defined
  by $I_\tau=I_\mu\setminus I_\nu$.

  This gives a one-to-one correspondence between the pairs $(\sigma, \tau)$ and the
  pairs $(\mu, \nu)$. Further observe that
  \[
    (-1)^{\sgn\tau}(\underbrace{\xi_{\tau,r}\land \xi_{\sigma,q}}_{\pm\xi_{\mu,
        q+r}})\otimes\xi^{\tau, r} = (-1)^{\sgn\mu}\xi_{\mu, q+r} \otimes
    (\underbrace{\xi_{\nu,q}\land\xi^{\mu,q+r}}_{\pm\xi^{\tau, r}}),
  \]
  so
  \begin{multline*}
    (-1)^{\sgn\tau+\sgn\sigma}(\xi_{\tau,r}\land\xi_{\sigma, q} \land
    \eta)\otimes \xi^{\sigma,q}\otimes\xi^{\tau,r} =\\
    (-1)^{\sgn\mu+\sgn\nu}(\xi_{\mu, q+r}\land\eta)\otimes\xi^{\nu,q}\otimes
    (\xi_{\nu,q}\land\xi^{\mu,q+r}).\qedhere
  \end{multline*}
\end{proof}

\subsection{Filtration diagrams}

Let $X$ be a scheme. As usual, $C(X)$ will denote the category of complexes of
$\sO_{X}$-modules and for $u\in\mor(C(X))$, $M(u)\in\ob(C(X))$ will denote the
mapping cone of $u$. $K(X)$ is the category of homotopy equivalence classes of
objects of $C(X)$. A diagram in $C(X)$ will be called a \emph{preditinguished
  triangle} if its image in $K(X)$ is a distinguished triangle.  $D(X)$ will denote
the derived category of complexes of $\sO_{X}$-modules. The superscripts $+, -, b$
carry the usual meaning (bounded below, bounded above and bounded). Regarding these
notions the basic reference will be \cite{Hartshorne66}.  $S_k$ denotes the symmetric
group of degree $k$.

\nsmorphd.

Let $p, i\in\bN$. We are going to define an object,
$\ff ip=\ff ip(\theta_X)\in\ob(C(X))$ and a \emph{$(p,i)$-filtration diagram of
  $\theta_X$} diagram, $\ffb ip=\boxed{\ff ip(\theta_X)}$.  This will be done
recursively, starting with $i=0$ and then increasing $i$.

\begin{definition}
  The \emph{$(p,0)$-filtration diagram of $\theta_X$} is
  \[
    \ffb 0p=\ff 0p= \Oxi {n-p}.
  \]
  A \emph{$0$-filtration morphism} for some $p, q$, consists of locally free sheaves
  $\sE, \sF$ and a morphism between $\Oxi {n-p}\otimes\sE$ and $\Oxi {n-q}\otimes\sF$.
\end{definition}
\noin
For instance,
\[
  \ww p\theta:\Oxi {n-p}.\tfoms \longrightarrow \Oxi n\tfosi {k-p}
\]
is a $0$-filtration morphism. Let
\[
  \ff 1p=M(\ww p\theta)[-1].
\]

\begin{definition}
  The \emph{$(p,1)$-filtration diagram of $\theta_X$} consists of the
  predistinguished triangle,
  \[
    \dist {\ff 1p}{\Oxi {n-p}\tfoms}{\Oxi n\tfosi {k-p}}
  \] 
  It is denoted by $\ffb 1p$. A \emph{$1$-filtration morphism} for some $p, r$,
  consists of locally free sheaves $\sE, \sF$ and morphisms between the corresponding
  terms of $\ffb 1p\otimes\sE$ and $\ffb 1r\otimes\sF$ such that the resulting
  diagram is commutative:
  \[
    \xymatrix{%
      {\ff 1p \otimes\sE} \ar[d]\ar[r] & {\ff 1r \otimes\sF} \ar[d]
      \\
      {\Oxi {n-p}\tfoms\otimes\sE} \ar[r]\ar[d]_{\ww p\theta} & {\Oxi
        {n-r}\tfoms\otimes\sF} \ar[d]^{\ww r\theta}
      \\
      {\Oxi n\tfosi {k-p}\otimes\sE} \ar[r] \ar[d]_{+1} & {\Oxi n\tfosi
        {k-r}\otimes\sF} \ar[d]^{+1} & \\ & .}
  \]
\end{definition}

\noin Consider the following commutative diagram (cf.~\autoref{wedgecomm}).
\begin{equation}
  \label{preegy}
  \begin{gathered}
    \xymatrix{%
      {\ff 1p\tfoms} \ar[d] & {\ff 1{p-q} \tfosi {k-q}} \ar[d]
      \\
      {\Oxi {n-p}\tfoms\tfoms} \ar[d]_{\ww p\theta} \ar[r]^-{\ww q\theta} & {\Oxi
        {n-p+q}\tfoms\tfosi {k-q}} \ar[d]^{\ww {p-q}\theta}
      \\
      {\Oxi n\tfosi {k-p}\tfoms} \ar[r]^-{\ww q.\id.}  \ar[d]_{+1} & {\Oxi n\tfosi
        {k-p+q}\tfosi {k-q}} \ar[d]^{+1} \\ & }
  \end{gathered}
\end{equation}

\noin There exists a morphism,
\[
  \alpha: \ff 1p\tfoms\to \ff 1{p-q}  \tfosi {k-q},
\]

\noin 
that makes the above diagram commutative. 

\noindent
The diagram \autoref{preegy}, combined with $\alpha$ gives a $1$-filtration morphism
\[
  \ffb 1p\tfoms \longrightarrow \ffb 1{p-q}\tfosi {k-q},
\]
with $r=p-q, \sE=\foms, \sF=\fosi {k-q}$.

Let
\[
  \ff 2p=M\bigg( {\ff 1p\tfoms}\to{\ff 11\tfosi {k-p+1}}\bigg)[-1].
\]
Then there exists a distinguished triangle,
\[
  \dist{\ff 2p}{\ff 1p\tfoms}{\ff 11\tfosi {k-p+1}}
\]

\begin{definition}
  The \emph{$(p,2)$-filtration diagram of $\theta_X$} consists of the diagram,
  \[
    \ff 2p \longrightarrow \ffb 1p\tfoms \longrightarrow \ffb 11\tfosi {k-p+1}.
  \]
  It is denoted by $\ffb 2p$. A \emph{$2$-filtration morphism} for some $p, r$,
  consists of locally free sheaves $\sE, \sF$ and morphisms between the corresponding
  terms of $\ffb 2p\otimes\sE$ and $\ffb 2r\otimes\sF$ such that the resulting
  diagram is commutative.
\end{definition}

\noin
More explicitly, the $(p,2)$-filtration diagram of $\theta_X$ is:
\[
  \xymatrix@C3em{%
    \ff 2p \ar[r] & \ar[d] {\ff 1p \tfoms} \ar[r] & \ff 11 \tfosi {k-p+1} \ar[d] \\
    &{\Oxi {n-p}\tfoms\tfoms} \ar[d]_{\ww p\theta}\ar[r]^-{\ww {p-1}\theta} & \Oxi
    {n-1}\tfoms\tfosi    {k-p+1} \ar[d]^{\ww 1\theta}\\
    & \Oxi n\tfosi {k-p}\tfoms \ar[r]^{\ww {p-1}\id} \ar[d]_{+1} & \Oxi
    n\tfosi {k-1}\tfosi {k-p+1} \ar[d]^{+1}\\
    && }
\]

\noin Similarly, a $2$-filtration morphism is:
\[
  \xymatrix@C1em{%
    \ffbi 2p{0,0}\!\otimes\sE\ar@{->}[rr] \ar[rd] && \ffbi 2p{0,1}
    \!\otimes\sE\ar@{->} [rr]\ar@{->}'[d] [dd] \ar[rd] && \ffbi 2p{0,2}
    \!\otimes\sE\ar@{->}'[d] [dd]\ar[rd]
    \\
        &\ffbi 2r{0,0}\!\otimes\sF\ar@{->}[rr] && \ffbi 2r{0,1} \!\otimes\sF\ar@{->}
    [rr]\ar@{->} [dd]
    && \ffbi 2r{0,2} \!\otimes\sF\ar@{->} [dd]\\
    && \ffbi 2p{1,1}\!\otimes\sE \ar[rd] \ar@{->}'[r][rr] \ar@{->} [dd] && \ffbi
    2p{1,2} \!\otimes\sE\ar[rd] \ar@{->}'[d] [dd]
    \\
    &&& \ffbi 2r{1,1}\!\otimes\sF\ar@{->}[rr] \ar@{->} [dd] && \ffbi 2r{1,2}
    \!\otimes\sF\ar@{->} [dd]
    \\
    && \ffbi 2p{2,1}\!\otimes\sE\ar[rd] \ar@{->}'[r][rr]
    && \ffbi 2p{2,2}\!\otimes\sE\ar[rd]
    \\
    &&& \ffbi 2r{2,1}\!\otimes\sF\ar@{->}[rr] && \ffbi 2r{2,2}\!\otimes\sF,
  }
\]

\noin
where the $(p,2)$-filtration diagram,
\[
  \xymatrix@R3em{%
    \ff 2p \ar@{->}[r] & {\ff 1p\tfoms} \ar@{->}[r]\ar@{->}[d]
    & \ff 11\tfosi {k-p+1}\ar@{->}[d]\\
    & \Oxi {n-p}\tfoms\tfoms \ar@{->}[r] \ar@{->}[d] & \Oxi {n-1}\tfoms\tfosi {k-p+1}
    \ar@{->}[d]
    \\
    & \Oxi n\tfosi {k-p}\tfoms \ar[d]^-{+1}\ar@{->}[r]
    & \Oxi n\tfosi {k-1}\tfosi {k-p+1}\ar[d]^-{+1}\\
    && }
\]

\noin
is represented by the simplified diagram,
\[
  \xymatrix@C5em{%
    \ffb 2p^{\ 0,0}\ar@{->}[r] & \ffb 2p^{\ 0,1} \ar@{->} [r]\ar@{->} @<-1.5ex> [d]
    & \ffb 2p^{\ 0,2}  \ar@{->} @<-1.5ex> [d]\\
    & \ffb 2p^{\ 1,1}\ar@{->}[r] \ar@{->} @<-1.5ex> [d] & \ffb 2p^{\ 1,2} \ar@{->}
    @<-1.5ex> [d]
    \\
    & \ffb 2p^{\ 2,1}\ar@{->}[r] & \ffb 2p^{\ 2,2}.  }
\]

To define the $(p,i)$-filtration diagram of $\theta_X$ and the
$i$-filtration morphisms we will iterate this construction.

\begin{num}{\bf Inductive Hypotheses.} \rm 
  \label{indstep}%
  For a given $i$ the following hold for each $p, q, r\in\bN$.
  \begin{enumerate}
  \item\label{indstepegy} The \emph{$(p,i)$-filtration diagram of $\theta_X$} is
    defined and denoted by $\ffb ip$.
  \item\label{indstepket} An \emph{$i$-filtration morphism}, by definition, consists
    of locally free sheaves $\sE, \sF$ and a morphism between the corresponding terms
    of $\ffb ip\otimes\sE$ and $\ffb ir\otimes\sF$ such that the resulting diagram is
    commutative.
  \item\label{indstephar} $\ffb ip$ has a unique object, $\ff ip$, with only one
    adjacent arrow pointing out.
  \item\label{indstepketfel} $\ff ip = 0$ for $p<i$.  
  \item\label{indstepnegy} There exists an $i$-filtration morphism,
    \[
      \ww q{\theta,i}:\ffb ip\tfoms \longrightarrow \ffb i{p-q}\tfosi {k-q}.
    \]
  \item\label{indstepot} The diagram,
    \[
      \xymatrix{%
        \ar[d]_{\ww {q+r}{\theta,i}}\ffb ip\otimes\foms\otimes\foms \ar[r]^{\ww
          q{\theta,i}} & \ffb i{p-q}\otimes \foms \otimes \fosi {k-q} \ar[d]^{\ww
          r{\theta,i}}        \\
        \ffb i{p-q-r}\otimes \fosi {k-q-r}\otimes \foms \ar[r]^{\ww q\id} & \ffb
        i{p-q-r}\otimes \fosi {k-r}\otimes \fosi {k-q} }
    \]
    is commutative.
  \end{enumerate}
\end{num}

\begin{lemma}\label{filtrationlemma}
  If \eqref{indstep} holds for $i=0,\dots, j$, then $\ffb {j+1}p$ can be defined so
  that \eqref{indstep} holds for $i=j+1$.
\end{lemma}

\begin{proof}
  For the proof the reader is referred to \cite{Kovacs05a}*{2.5}.
\end{proof}

\noin Now we are ready to define $\Oxsui p\in\ob(D(X))$.

\begin{definition}\label{def:Q}
  Let $p\in\bZ$. For $p>n-k$ let $\Oxsui p=0$, and for $-k\leq p\leq n-k$
  let $\Oxsui p$ be the class of
  \[
    \ff {n-k-p}{n-k-p}\otimes(\foms)^{-(n-k-p+1)}
  \]
  in $\ob(D(X))$. It follows that
  \[
    \Oxsui {n-k}=\det\Psi_X\otimes(\foms)^{-1}
  \]
  and that there is a distinguished triangle:
  \[ 
    \dist {\Oxsui {n-k-1}\otimes\foms}{\Oxi {n-1}}{\Oxsui {n-k}\otimes\fosi {k-1}}.
  \]
  Furthermore, for $j\geq p-n+k$ let $\hypf j\Oxi p$ be the class of
  \[
    \ff {n-k-p+j}{n-p}\otimes(\foms)^{-(n-k-p+j)}
  \]
  in $\ob(D(X))$.  The predistinguished triangle,
\[
  \dist {\ff {n-k-p+j+1}{n-p}}{\ff {n-k-p+j}{n-p}\tfoms}{\ff {n-k-p+j}{n-k-p+j}\tfosi
    j}
\]
obtained during the definition of these filtration diagrams
(cf.\cite{Kovacs05a}*{(2.5.7)}) gives the distinguished triangle,
\[ 
  \dist {\hypf{j+1}\Oxi p}{\hypf j\Oxi p}{\Oxsui {p-j}\tfosi j}.
\]
Observe that $\hypf{k+1}\Oxi p=0$ by \autoref{indstepketfel} and
$\hypf{p-n+k}\Oxi p=\Oxi p$ by definition.  Furthermore, if $p-n+k<0$, then
\[ 
  \hypf{0}\Oxi p\simeq\hypf{-1}\Oxi p\simeq\dots\simeq \hypf{p-n+k}\Oxi p=\Oxi p,
\] 
because $\fosi j=0$ for $j<0$. If $p-n+k\geq 0$, define $\hypf{j}\Oxi p=\Oxi p$ for
$j=0,\dots, p-n+k$.
\end{definition}

\noin
The following theorem summarizes the above observations.

\begin{theorem}\cite{Kovacs05a}*{2.7}\label{spseqegy}
  \nsmorphd.  Then there exists an object $\Oxsui r\in\ob(D(X))$ for each
  $r\in\bZ, r\geq -k$ with the following property. For each $p\in\bN$ there exists a
  hyperfiltration $\hypf j\Oxi p$ of $\Oxi p$ with $j=0,\dots, k+1$, such that
  \[
    \hypf{0}\Oxi p\simeq\Oxi p,
  \]
  \[
    \hypf{k+1}\Oxi p\simeq 0\] and 
  \[
    \bG r^j\Oxi p\simeq\Oxsui {p-j}\tfosi j.
  \]
  Furthermore, for $r> n-k$,
  \[
    \Oxsui r\simeq 0.
  \]
\end{theorem}

\begin{proposition}\cite{Kovacs05a}*{2.9}\label{easy}
  Assume that $\theta_X$ is injective. Then if $\Xi_X$, the cokernel of $\theta_X$,
  is locally free, then $\Oxsui p$ is isomorphic to the $p$-th exterior power of
  $\Xi_X$.  The filtration is given by
  \[ 
    \Oxi p=F^0\supset F^1 \supset \dots \supset F^p \supset F^{p+1}=0,
  \] 
  with quotients
  \[ 
    F^j/F^{j+1}\simeq \bigwedge^{p-j}\Xi_X\tfosi j  ,
  \] 
  for each $j$.  
\end{proposition}

\begin{proof}
  By definition one has that
  \[
    \Oxsui {n-k}\simeq \det\Psi_X\otimes (\det\Phi_X)^{-1}\simeq \det\Xi_X.
  \]
  Then the statement follows using descending induction, the filtration associated to
  the short exact sequence of locally free sheaves and the distinguished triangle,
  \[
    \dist {\hypf{j+1}\Oxi p}{\hypf j\Oxi p}{\Oxsui {p-j}\tfosi j}.\qedhere
  \]
\end{proof}

\begin{example}
  Let $k=1$, i.e., assume that $\Phi_X$ is a line bundle. Then the hyperfiltration in
  \autoref{spseqegy} is simply a distinguished triangle
  \[
    \dist{\Oxsui {p-1}\otimes\Phi_X}{\Oxi p}{\Oxsui p}.
  \]
  \end{example}

\bibliographystyle{bibliography/skalpha.bst} 
\bibliography{bibliography/general}


\end{document}


It was observed by Viehweg that for a projective morphism $f:X\to B$ of smooth
projective varieties $X$ and $B$ with connected fibers, the positivity in direct
image sheaves $f_*\omega^m_{X/B}$ is closely related to the positivity of
$\omega_B(D)$, with $D$ being the discriminant locus of $f$. Specifically, he
conjectured that bigness of $\det f_*\omega^m_{X/B}$ (assuming that it is preserved
under semistable reduction in codimension one) implies bigness of
$\omega_B(D)$.\!\footnote{The original conjecture was formulated in terms of
  variation, when smooth fibers are canonically polarized.}  This conjecture
generated considerable interest and for several years. It was finally resolved, and
in fact generalized, through the culmination of work of several authors including
\cite{Kovacs96}, \cite{Kovacs00a}, \cite{Oguiso-Viehweg01}, \cite{Vie-Zuo01},
\cite{Kovacs02}, \cite{VZ02}, \cite{Kovacs03b}, \cite{KK08},\cite{KK08b},
\cite{KK10}, \cite{MR2871152}, \cite{PS15}, \cite{Taj18} and \cite{CP16}.

In higher dimensions, the minimal model program taught us that when 
 positivity of canonical sheaves are involved, it is desirable to try to extend results
to mildly singular cases. So, it is natural to ask whether Viehweg's conjecture
extends to families of minimal models. 
The simple answer is that the desired positivity fails already
if one allows Gorenstein terminal singularities, arguably the mildest possible. In
particular, the conjecture fails for Lefschetz pencils,
cf.~\ref{sssec:order-poles-lefsch}.

This could be interpreted as a sign that there is no reasonable generalization of
Viehweg's conjecture to singular families. However, here we offer a potential way to
remedy the situation. More particularly, as an application of \autoref{thm:DBI'}, we
show that for Gorenstein families it is possible to obtain a result similar to
Viehweg's conjecture. This requires that we take into account how singular the family
is. More precisely, we show that if $\det f_*\omega_{X/B}$ is positive \emph{enough}
to balance the \emph{discrepancy} of the family (discussed above), then the base of
the family is indeed necessarily of general type. Notice that this is in fact a
direct generalization of Viehweg's conjecture, i.e., if $f$ is smooth, then
\autoref{thm:Cons} below reduces to Viehweg's conjecture.